%% file: arxiv-On-Riemann-curvature-of-singular-square-metrics.tex
\documentclass[10pt,a4paper]{article}
\usepackage{amsmath}
\usepackage{latexsym,amssymb,amsmath,amsthm,amsfonts}
\usepackage{graphicx}
\usepackage{enumerate}
\usepackage{color}

\allowdisplaybreaks[4]

\newtheorem{lemma}{Lemma}[section]
\newtheorem{proposition}[lemma]{Proposition}
\newtheorem{theorem}[lemma]{Theorem}

\newenvironment{remark}{\textbf{Remark}}{}

\numberwithin{equation}{section}

\input{symbol}

\makeatletter

\newcommand{\Rmnum}[1]{\expandafter\@slowromancap\romannumeral #1@}
\makeatother

\begin{document}
\title{On Riemann curvature of singular square metrics}

\author{Changtao Yu}

\date{2018.07.17}
\maketitle

\begin{abstract}
Square metrics is an important class of Finsler metrics. Recently, we introduced a special class of non-regular Finsler metrics called singular square metrics.  The main purpose of this paper is to provide a necessary and sufficient condition for singular square metrics to be of constant Ricci or flag curvature when dimension $n\geq3$.
\end{abstract}

\section{Introduction}
Square metrics $F=\f{(\a+\b)^2}{\a}$ are a special class of $\ab$-metrics~($\ab$-metrics are the Finsler metrics expressed as $F=\pab$, where $\phi$  is a smooth function, $\a$ is a Riemannian metric and $\b$ is a $1$-form). The first square metric in history is the famous Berwald's metric\cite{Be}
\begin{eqnarray*}
F=\f{(\sqrt{(1-|x|^2)|y|^2+\langle x,y\rangle^2}+\langle x,y\rangle)^2}{(1-|x|^2)^2\sqrt{(1-|x|^2)|y|^2+\langle x,y\rangle^2}},
\end{eqnarray*}
which is defined on the unit ball $\mathbb B^n(1)$ with all the straight line segments as its geodesics and is of constant flag curvature $K=0$.

In Finsler geometry, square metrics is the rate kind of metrical category to be of excellent geometrical properties. In 2007, Li-Shen prove that, except for the trivial case, any locally projectively flat $\ab$-metric with constant flag curvature is either a Randers metric or a square metric\cite{lb-szm-onac}. In 2012, Cheng-Tian and Sevim-Shen-Zhao prove independently that, except for the trivial case, any Douglas-Einstein $\ab$-metric is either a Randers metric or a square metric\cite{cxy-tyf,sevim-szm-zll}.

In 2014, Z.Shen and the author classified the Einstein square metric up to the classification of Einstein Riemannian metrics by using some particular $\b$-deformations tailored for square metrics\cite{szm-yct-oesm}. Our result is similar to the famous classification for Randers metrics with constant flag or Ricci curvature based on Zermelo's navigation problem\cite{db-robl-szm-zerm,db-robl-onri}.

All the results mentioned above merely apply to {\it regular} square metrics. It is known that a square metric $F=\f{(\a+\b)^2}{\a}$ is a regular Finsler metric if and only if the length of $\b$ with respect to $\a$, denoted by $b$, satisfies $b<1$.

Recently, we studied the flag curvature of general $\ab$-metrics $F=\gab$, which includes $\ab$-merics naturally\cite{yct-zhm-onan}. In such bigger metrical category, we found many new metrics which are locally projectively flat and of constant flag curvature beyond Randers metrics and (regular) square metrics\cite{yct-zhm-pfga}. Among these new class of metrics, one is given by
\begin{eqnarray}\label{singularBerwald}
F=\f{(b\a+\b)^2}{\a},
\end{eqnarray}
which can be also rewrote as $F=\f{(\breve\a+\breve\b)^2}{\breve\a}$ where $\breve\a=b^2\a$ and $\breve\b=b\b$. Obviously, the length of $\breve\b$ with respect to $\breve\a$, denoted by $\breve b$, satisfies $\breve b=\sup\f{\breve\b}{\breve\a}=\sup\f{\b}{b\a}=1$. Hence, (\ref{singularBerwald}) are just the square metrics $F=\sq$ with $b\equiv1$. Conversely, any square metric $F=\sq$ with $b\equiv1$ can be expressed as the form (\ref{singularBerwald}). See \cite{ossm} for details.

We will call the non-regular Finsler metrics (\ref{singularBerwald}) as {\it singular square metrics}. Although they are equivalent to square metrics $F=\sq$ with $b\equiv1$, we believe that the expression (\ref{singularBerwald}) is better. The main reason is that $b$ is {\it not necessary} a constant. This is important if one would like to discuss such metrics with $\b$-deformations\cite{yct-dhfp}. As an example, in \cite{ossm} we characterized Douglas singular square metrics and provided an incomplete classification by using $\b$-deformations.

The main results of this paper are the following characterizations for singular square metrics with constant Ricci or flag curvature.

\begin{theorem}\label{iwneingadad}
Let $F=\f{(b\a+\b)^2}{\a}$ be a singular square metric on a $n$-dimensional manifold with $n\geq3$. Then $F$ is an Einstein metric if and only if it is Ricci-flat, $\a$ and $\b$ satisfy
\begin{eqnarray}
\Ricoo&=&-\f{1}{9b^4}\big\{[9(n-1)b^2c^2-2b^6d^2-6b^4\db+144(n-3)t+36 b^2\sIhi]\a^2-2(n-2)b^4d^2\b^2\nonumber\\
&&-6(n-2)b^4\b d_0-72(n-2)c\b\so-84(n-2)b^2d\b\so+288(n-2)\so^2+36(n-2)b^2\soho\big\},\label{bewaldcondiononalpha}\\
\roo&=&c\a^2+d\b^2-\f{6}{b^2}\b\so,\label{bewaldcondiononbeta}\\
\sIo&=&\f{1}{b^2}(\so b^i-\b s^i),\label{bewaldcondiononbetas}
\end{eqnarray}
where $c=c(x)$ and $d=d(x)$ are scalar functions.
\end{theorem}
\begin{remark}
$\ck:=c_{|k}$, $\cI:=a^{ik}\ck$, $\co:=\ck y^k$, $\cb:=\ck b^k$~($\dk$ has similar rules).
\end{remark}

\begin{remark}
If $\a$ and $\b$ satisfy (\ref{bewaldcondiononalpha})-(\ref{bewaldcondiononbetas}), then $\breve\a=b^2\a$ and $\breve\b=b\b$ satisfy
\begin{eqnarray*}
{}^{\breve{\alpha}}Ric_{00}&=&-\f{1}{3}\breve\tau^2\left\{(3n-5)\breve\a^2-2(n-2)\breve\b^2\right\},\\
\breve b_{i|j}&=&\breve\tau(\breve a_{ij}-\breve b_i\breve b_j),\\
\breve\tau_i&=&-\f{2}{3}\breve\tau^2\breve b_i,
\end{eqnarray*}
where $\breve\tau=\frac{3c+2b^2d}{b^3}$. In this case, the metric is given by $F=\frac{(\breve\a+\breve\b)^2}{\breve\a}$ with $\breve b\equiv1$. It matches with Theorem 2.1 in \cite{szm-yct-oesm}. But we can not obtain a conclusion similar to Theorem 1.1 (see also Lemma 2.3) in \cite{szm-yct-oesm} since $\bar b\equiv1$.
\end{remark}

\begin{theorem}\label{odnanganndnand}
Let $F=\f{(b\a+\b)^2}{\a}$ be a singular square metric on a $n$-dimensional manifold with $n\geq3$. Then $F$ has constant flag curvature $K$  if and only if $K=0$, $\a$ and $\b$ satisfy
\begin{eqnarray}
\RIk&=&-\f{1}{9b^4}\big\{[(9b^2c^2+144t)\a^2-2b^4d^2\b^2-6(b^4d_0+12c\so+14b^2d\so)\b+288\so^2+36b^2\soho]\dIk\nonumber\\
&&-(9b^2c^2+144t)\yI\yk+2b^2(b^2d^2\b+3b^2d_0+6d\so)(\yI\bk+\yk\bI)-2b^4d^2\a^2\bI\bk-3b^4\a^2(\bI\dk+\bk d^i)\nonumber\\
&&-6\a^2(6c+7b^2d)(\bI\sk+\bk\sI)+72[(c+b^2d)\b-4\so](\yI\sk+\yk\sI)+288\a^2\sI\sk-36b^2(\yI\skho+\yk\sIho)\nonumber\\
&&+18b^2\a^2(\sIhk+\skhI)\big\},\label{bewaldcondiononalphaflag}\\
\rij&=&c\aij+d\bi\bj-\f{3}{b^2}(\bi\sj+\bj\si),\label{bewaldcondiononbetarijflag}\\
\sij&=&\f{1}{b^2}(\bi\sj-\bj\si),\label{bewaldcondiononbetasijflag}
\end{eqnarray}
where $c=c(x)$ and $d=d(x)$ are scalar functions.
\end{theorem}

As an application, we can obtain a classification for a simple case.

Assume $\b$ is conformal with respect to $\a$, then by (\ref{bewaldcondiononbeta}) we have $d=0$ and $\so=0$, and by (\ref{bewaldcondiononbetas}), $\b$ is closed. Moreover, (\ref{woeminaddd'amd}) indicates $b^2\co=c^2\b$, which is equivalent to $(\ln c^2)_{k}=(\ln b^2)_k$. Hence, there exist a non-positive constant $\mu$ such that $c^2=-\mu b^2$. Plugging $c$ into (\ref{bewaldcondiononalpha}) yields $\Ricoo=(n-1)\mu\a^2$, and plugging $c$ into (\ref{bewaldcondiononalphaflag}) yields $\RIk=\mu(\a\dIk-y^iy_k)$. As a result, we have the following classification.
\begin{theorem}
Let $F=\f{(b\a+\b)^2}{\a}$ be a singular square metric on a $n$-dimensional manifold with $n\geq3$. If $\b$ is conformal with respect to $\a$, then $F$ has constant Ricci (resp. flag) curvature if and only if $\a$ has constant Ricci (resp. sectional) curvature $\mu$ and $\b$ is closed with the conformal factor $c(x)$ satisfying $c^2=-\mu b^2$.
\end{theorem}

As a non-trivial example, taking
\begin{eqnarray*}
\a=\f{\sqrt{(1+\mu|x|^2)|y|^2-\mu\langle x,y\rangle^2}}{1+\mu|x|^2},\qquad
\b=\f{\lambda\langle x,y\rangle+(1+\mu|x|^2)\langle a,y\rangle-\mu\langle a,x\rangle\langle x,y\rangle}{(1+\mu|x|^2)^\frac{3}{2}}.
\end{eqnarray*}
with the constant number $\lambda$ and constant vector $a$ satisfying $\lambda^2+\mu|a|^2=0$ additionally, then $\a$ is of negative curvature $\mu$ and $\b$ is closed and conformal with the conformal factor $c(x)$ satisfying $c^2=-\mu b^2$. Hence, singular square metric (\ref{singularBerwald}) is of vanishing flag curvature when $n\geq3$. This example was first provided in \cite{yct-zhm-pfga}.

\section{Preliminaries}
The Riemann tensor ${}^FR_{y}=\RIkF\frac{\partial}{\partial x^{i}}\otimes
dx^{k}$ of a Finsler $F(x,y)$ is determined by Berwald's formula as bellow,
\begin{eqnarray*}
\RIkF:=2G^{i}_{x^{k}}-G^{i}_{x^{l}y^{k}}y^{l}+2G^{l}G^{i}_{y^{k}y^{l}}-G^{i}_{y^{l}}G^{l}_{y^{k}},
\end{eqnarray*}
where $G^{i}:=\frac{1}{4}g^{il}\left\{[F^{2}]_{x^{k}y^{l}}y^{k}-[F^{2}]_{x^{l}}\right\}$ are spray coefficients of $F$, in which $(g^{ij})$ is the inverse of the fundamental tensor $g_{ij}:=\frac{1}{2}[F^2]_{y^iy^j}$. $F$ is said to be of constant flag curvature $K$ if $\RIkF=K(F^2\dIk-Fy^iF_{y^k})$.

The Ricci tensor of $F$ due to Akbar-Zadeh is given by
\begin{eqnarray*}
{}^F{Ric}_{ij}:=\left(\f{1}{2}{}^FR^m{}_m\right)_{y^iy^j}.
\end{eqnarray*}
Denote $\RicooF:={}^F{Ric}_{ij}y^iy^j$. $F$ is said to be of constant Ricci curvature $K$ if $\RicooF=(n-1)KF^2$.

Using the Maple program, we can calculate the Riemann tensor and Ricci tensor of singular square metrics. The whole process is elementary and mechanical, so we provide the final results in appendix directly.

Moreover, we need some abbreviations. Let $\a=\sqrt{a_{ij}y^iy^j}$ be a Riemannian metric, $\b=b_iy^i$ be a $1$-form. Denote
\begin{eqnarray*}
\rij:=\f{1}{2}(\bij+b_{j|i}),\quad\sij:=\f{1}{2}(\bij-b_{j|i})
\end{eqnarray*}
be the symmetrization and antisymmetrization of the covariant derivative $\bij$ respectively, then
\begin{eqnarray*}
&r_{i0}:=r_{ij}y^j,\quad r^i{}_0:=a^{im}r_{m0},\quad r_{00}:=r_{i0}y^i,\quad r_i:=b^mr_{mi},\quad r^i:=a^{im}r_{m},\quad r_0:=r_iy^i,\quad r:=r_ib^i,\\
&s_{i0}:=s_{ij}y^j,\quad s^i{}_0:=a^{im}s_{m0},\quad s_i:=b^ms_{mi},\quad s^i:=a^{im}s_{m},\quad s_0:=s_iy^i.
\end{eqnarray*}
Roughly speaking, indices are raised or lowered by $a^{ij}$ or $\aij$, vanished by contracted with $b^i$~(or~$\bi$) and changed to be '${}_0$' by contracted with $y^i$ (or $y_i:=a_{ij}y^j$). At the same time, we need other three tensors
\begin{eqnarray*}
p_{ij}:=r_{im}r^m{}_j,\qquad q_{ij}:=r_{im}s^m{}_j,\qquad t_{ij}:=s_{im}s^m{}_j
\end{eqnarray*}
and the related tensors determined by the above rules. Notice that both $p_{ij}$ and $t_{ij}$ are symmetric, but $q_{ij}$ is neither symmetric nor antisymmetric in general. So $b^jq_{ji}$ and $b^jq_{ij}$, denoted by $q_i$ and $q^\star_i$ respectively, are different. But $b^iq_i$, denoted by $q$, is equal to $b^iq^\star_i$. Finally, in order to avoid ambiguity, sometimes we will use index $b$, which means contracting the corresponding index with $b^i$ or $b_i$. For example, $\rohb:=r_{0|k}b^k$.

Finally, the following priori formulae for $\b$ are required. These formulae show the inner relationship between the covariant derivative of $\sij~(\rij)$ and the Riemann tensor. They hold for any Riemann metric $\a$ and any $1$-form $\b$ without any restriction. The key formula (\ref{relation1}) play a crucial role in the study of Randers metrics \cite{db-robl-onri}.
\begin{lemma}\cite{xy}
\begin{eqnarray}\label{3ingaindngiad}
r_{i|j}+s_{i|j}=r_{j|i}+s_{j|i},
\end{eqnarray}
as a corollary,
\begin{eqnarray}
\sohb=-p_{0}-q_{0}+q^{*}_{0}-t_{0}-\rohb+r_{|0}.\label{relation5}
\end{eqnarray}
\end{lemma}

\begin{lemma}\cite{xy}
\begin{eqnarray}\label{relation1}
s_{ij|k}=-b^m R_{kmij}+r_{ik|j}-r_{jk|i},
\end{eqnarray}
where $R_{kmij}$ is the fourth-order Riemann tensor determined by $R_{kmij}=\frac{1}{3}\left(\frac{\partial^2R_{mi}}{\partial y^j\partial y^k}-\frac{\partial^2R_{mj}}{\partial y^i\partial y^k}\right)$. As corollaries,
\begin{eqnarray}
\sIkho&=&\RoIkb-\RbIko-\rkohI+\rIohk,\label{relation7}\\
\sIohk&=&\RoIkb+\rIkho-\rkohI,\label{relation8}\\
\sIoho&=&-\RIb+\rIoho-\roohI,\label{relation9}\\
\sIohb&=&-\qIo+\qoI+\rIho-\rohI,\label{relation14}\\
\sIhk&=&-\RbIkb-\pIk-\tIk-\rIkhb+\rkhI,\label{relations10}\\
\skho&=&-\Rbokb-\pko-\tko-\rkohb+\rohk,\label{relation12}\\
\sohk&=&-\Rbokb-\pko-\tko-\rkohb+\rkho,\label{relation13}\\
\soho&=&-\Rbb-p_{00}-t_{00}-r_{00|b}+r_{0|0},\label{relation4}\\
\sIohi&=&\Rico+\rIiho-\rIohi,\label{relation2}\\
\sIhi&=&-\Ric-\pIi-\tIi-\rIihb+\rIhi.\label{relation6}
\end{eqnarray}
\end{lemma}

\section{Constant Ricci curvature:~necessary and sufficient conditions}
Assume singular square metric (\ref{singularBerwald}) is an Einstein metric with Ricci scalar $K(x)$, namely
\begin{eqnarray*}
\RicooF-(n-1)KF^2=0.
\end{eqnarray*}
Combining with (\ref{yanega;dngadg}) in appendix, the above equality is equivalent to
\begin{eqnarray}\label{pwpeiming}
b\Rat+\a\Irrat=0,
\end{eqnarray}
where $\Rat$ and $\Irrat$ are polynomials on $y$, given by
\begin{eqnarray}
\Rat&=&\mathfrak{R}_1\Ricoo+\mathfrak{R}_2K+\mathfrak{R}_3\roo^2+\mathfrak{R}_4\rooho+\mathfrak{R}_5\roo\ro+\mathfrak{R}_6\roo\so
+\mathfrak{R}_7\roo\rIi+\mathfrak{R}_8\roo r+\mathfrak{R}_9\qoo+\mathfrak{R}_{10}\too\nonumber\\
&&+\mathfrak{R}_{11}\ro^2+\mathfrak{R}_{12}\ro\so
+\mathfrak{R}_{13}\so^2+\mathfrak{R}_{14}\roohb+\mathfrak{R}_{15}\roho+\mathfrak{R}_{16}\soho+\mathfrak{R}_{17}\ro\rIi
+\mathfrak{R}_{18}\so\rIi+\mathfrak{R}_{19}\ro r\nonumber\\
&&+\mathfrak{R}_{20}\so r+\mathfrak{R}_{21}\po+\mathfrak{R}_{22}\qqo+\mathfrak{R}_{23}\qo+\mathfrak{R}_{24}\to+\mathfrak{R}_{25}\sIohi+\mathfrak{R}_{26}\rohb
+\mathfrak{R}_{27}\sohb+\mathfrak{R}_{28}r_{|0}+\mathfrak{R}_{29}r\rIi\nonumber\\
&&+\mathfrak{R}_{30}r^2+\mathfrak{R}_{31}\tIi+\mathfrak{R}_{32}p+\mathfrak{R}_{33}q+\mathfrak{R}_{34}t+\mathfrak{R}_{35}\rIhi
+\mathfrak{R}_{36}\sIhi+\mathfrak{R}_{37}\rhb,\label{RatBerwald}\\
\Irrat&=&\mathfrak{I}_1\Ricoo+\mathfrak{I}_2K+\mathfrak{I}_3\roo^2+\mathfrak{I}_4\rooho+\mathfrak{I}_5\roo\ro+\mathfrak{I}_6\roo\so
+\mathfrak{I}_7\roo\rIi+\mathfrak{I}_8\roo r+\mathfrak{I}_9\qoo+\mathfrak{I}_{10}\too\nonumber\\
&&+\mathfrak{I}_{11}\ro^2+\mathfrak{I}_{12}\ro\so+\mathfrak{I}_{13}\so^2+\mathfrak{I}_{14}\roohb
+\mathfrak{I}_{15}\roho+\mathfrak{I}_{16}\soho+\mathfrak{I}_{17}\ro\rIi+\mathfrak{I}_{18}\so\rIi+\mathfrak{I}_{19}\ro r+\mathfrak{I}_{20}\so\nonumber\\
&& r+\mathfrak{I}_{21}\po+\mathfrak{I}_{22}\qqo+\mathfrak{I}_{23}\qo+\mathfrak{I}_{24}\to+\mathfrak{I}_{25}\sIohi+\mathfrak{I}_{26}\rohb
+\mathfrak{I}_{27}\sohb+\mathfrak{I}_{28}r_{|0}+\mathfrak{I}_{29}r\rIi+\mathfrak{I}_{30}r^2\nonumber\\
&&+\mathfrak{I}_{31}\tIi+\mathfrak{I}_{32}p+\mathfrak{I}_{33}q+\mathfrak{I}_{34}t+\mathfrak{I}_{35}\rIhi+\mathfrak{I}_{36}\sIhi+\mathfrak{I}_{37}\rhb,\label{IrratBerwald}
\end{eqnarray}
where
\begin{eqnarray*}
&\mathfrak{R}_1=9b^2\a^2(b^2\a^2-\b^2)^2(b^2\a^2+\b^2),\quad\mathfrak{R}_2=-9(n-1)b^2(b^2\a^2-\b^2)^4(b^2\a^2+\b^2),\\
&\mathfrak{R}_3=b^2\a^2\{(3n-5)b^4\a^4+(29n-41)b^2\a^2\b^2+10(n-2)\b^4\},\\
&\mathfrak{R}_4=6b^2\a^2\b(b^2\a^2-\b^2)\{(2n-3)b^2\a^2+(n-2)\b^2\},\quad\mathfrak{R}_5=-2b^2\\
&\mathfrak{R}_6=-6(n-3)b^2\a^4\b(b^2\a^2-\b^2),\quad\mathfrak{R}_7=6b^2\a^4(b^2\a^2-\b^2)(b^2\a^2+\b^2),\\
&\mathfrak{R}_8=2b^2\a^6\{(9n-25)b^2\a^2+(33n-41)\b^2\},\quad\mathfrak{R}_9=12b^2\a^4(b^2\a^2-\b^2)\{(2n-1)b^2\a^2+(4n-7)\b^2\},\\
&\mathfrak{R}_{10}=-36b^2\a^4(b^2\a^2-\b^2)(b^2\a^2-3\b^2),\quad\mathfrak{R}_{11}=\a^4\{2(21n-64)b^4\a^4+(145n-143)b^2\a^2\b^2-9(n-1)\b^4\},\\
&\mathfrak{R}_{12}=6\a^4\{2(n-14)b^4\a^4+(11n+27)b^2\a^2\b^2-3(n-1)\b^4\},\\
&\mathfrak{R}_{13}=9\a^4\{2(n-8)b^4\a^4+(9n+17)b^2\a^2\b^2-(n-1)\b^4\},\quad\mathfrak{R}_{14}=6b^2\a^4(b^2\a^2-\b^2)(b^2\a^2+\b^2),\\
&\mathfrak{R}_{15}=-6b^2\a^4(b^2\a^2-\b^2)\{(2n-5)b^2\a^2+(4n-5)\b^2\},\quad\mathfrak{R}_{16}=18b^2\a^4(b^2\a^2-\b^2)^2,\\
&\mathfrak{R}_{17}=-24b^2\a^6\b(b^2\a^2-\b^2),\quad\mathfrak{R}_{18}=0,\quad\mathfrak{R}_{19}=-4\a^6\b\{(41n-101)b^2\a^2+6(n+6)\b^2\},\\
&\mathfrak{R}_{20}=-12\a^6\b\{(3n-11)b^2\a^2+2(n+6)\b^2\},\quad\mathfrak{R}_{21}=36(n-1)b^2\a^6\b(b^2\a^2-\b^2),\\
&\mathfrak{R}_{22}=-12(3n-5)b^2\a^6\b(b^2\a^2-\b^2),\quad\mathfrak{R}_{23}=-12(5n-11)b^2\a^6\b(b^2\a^2-\b^2),\\
&\mathfrak{R}_{24}=36(n-1)b^2\a^6\b(b^2\a^2-\b^2),\quad\mathfrak{R}_{25}=-36b^2\a^4\b(b^2\a^2-\b^2)^2,
\quad\mathfrak{R}_{26}=-24b^2\a^6\b(b^2\a^2-\b^2),\\
&\mathfrak{R}_{27}=0,\quad\mathfrak{R}_{28}=6(3n-5)b^2\a^6\b(b^2\a^2-\b^2),\quad\mathfrak{R}_{29}=12b^2\a^8(b^2\a^2-\b^2),\\
&\mathfrak{R}_{30}=2\a^8\{(12n-55)b^2\a^2+(14n+23)\b^2\},\quad\mathfrak{R}_{31}=-36b^2\a^6(b^2\a^2-\b^2)^2,\\
&\mathfrak{R}_{32}=-6\a^6(b^2\a^2-\b^2)\{(4n-17)b^2\a^2+3(n+4)\b^2\},\\
&\mathfrak{R}_{33}=12\a^6(b^2\a^2-\b^2)\{(2n-17)b^2\a^2+3(n+4)\b^2\},\quad\mathfrak{R}_{34}=-18\a^6(b^2\a^2-\b^2)\{3b^2\a^2-(n+4)\b^2\},\\
&\mathfrak{R}_{35}=-18b^2\a^6(b^2\a^2-\b^2)^2,\quad\mathfrak{R}_{36}=-18b^2\a^6(b^2\a^2-\b^2)^2,\quad
\mathfrak{R}_{37}=12b^2\a^8(b^2\a^2-\b^2),\\
&\mathfrak{I}_1=-18b^4\a^2\b(b^2\a^2-\b^2)^2,\quad\mathfrak{I}_2=-18(n-1)b^4\b(b^2\a^2-\b^2)^4,\\
&\mathfrak{I}_3=-2b^4\a^2\b\{(7n-9)b^2\a^2+2(7n-12)\b^2\},\quad\mathfrak{I}_4=-3b^4\a^2(b^2\a^2-\b^2)\{(n-1)b^2\a^2+(5n-9)\b^2\},\\
&\mathfrak{I}_5=2b^2\a^2\{(9n-5)b^4\a^4+2(33n-59)b^2\a^2\b^2+9(n-1)\b^4\},\\
&\mathfrak{I}_6=-6b^2\a^2(b^2\a^2-\b^2)\{(n-5)b^2\a^2+3(n-1)\b^2\},\quad\mathfrak{I}_7=-12b^4\a^4\b(b^2\a^2-\b^2),\\
&\mathfrak{I}_8=-2b^2\a^4\b\{(25n-47)b^2\a^2+(17n-19)\b^2\},\quad\mathfrak{I}_9=-24(3n-4)b^4\a^4\b(b^2\a^2-\b^2),\\
&\mathfrak{I}_{10}=72b^4\a^4\b(b^2\a^2-\b^2),\quad\mathfrak{I}_{11}=-\a^2\b\{(95n-167)b^4\a^4+36(2n-3)b^2\a^2\b^2-9(n-1)\b^4\},\\
&\mathfrak{I}_{12}=6\a^2\b\{(19n-43)b^4\a^4-12(n-4)b^2\a^2\b^2+3(n-1)\b^4\},\\
&\mathfrak{I}_{13}=9\a^2\b\{(9n-17)b^4\a^4+20b^2\a^2\b^2+(n-1)\b^4\},\quad\mathfrak{I}_{14}=-12b^4\a^4\b(b^2\a^2-\b^2),\\
&\mathfrak{I}_{15}=3b^2\a^2\b(b^2\a^2-\b^2)\{(9n-17)b^2\a^2+3(n-1)\b^2\},\quad\mathfrak{I}_{16}=-9(n-1)b^2\a^2\b(b^2\a^2-\b^2)^2,\\
&\mathfrak{I}_{17}=6b^2\a^4(b^2\a^2-\b^2)(b^2\a^2+3\b^2),\quad\mathfrak{I}_{18}=-18b^2\a^4(b^2\a^2-\b^2)^2,\\
&\mathfrak{I}_{19}=2\a^4\{(18n-31)b^4\a^4+62(n-2)b^2\a^2\b^2-3(2n-7)\b^4\}),\\
&\mathfrak{I}_{20}=-6\a^4\{(6n-35)b^4\a^4+2(n+22)b^2\a^2\b^2+(2n-7)\b^4\},\\
&\mathfrak{I}_{21}=-6b^2\a^4(b^2\a^2-\b^2)\{(2n-1)b^2\a^2+(4n-5)\b^2\},\\
&\mathfrak{I}_{22}=6b^2\a^4(b^2\a^2-\b^2)\{(2n-5)b^2\a^2+(4n-5)\b^2\},\\
&\mathfrak{I}_{23}=6b^2\a^4(b^2\a^2-\b^2)\{(8n-37)b^2\a^2+3(2n+9)\b^2\},\quad\mathfrak{I}_{24}=-18b^2\a^4(b^2\a^2-\b^2)\{11b^2\a^2-(2n+9)\b^2\},\\
&\mathfrak{I}_{25}=36b^4\a^4(b^2\a^2-\b^2)^2,\quad\mathfrak{I}_{26}=6b^2\a^4(b^2\a^2-\b^2)(b^2\a^2+3\b^2),\quad
\mathfrak{I}_{27}=-18b^2\a^4(b^2\a^2-\b^2)^2,\\
&\mathfrak{I}_{28}=-6b^2\a^4(b^2\a^2-\b^2)\{(n-1)b^2\a^2+2(n-2)\b^2\},\quad\mathfrak{I}_{29}=-12b^2\a^6\b(b^2\a^2-\b^2),\\
&\mathfrak{I}_{30}=-4\a^6\b\{4(2n-5)b^2\a^2+3\b^2\},\quad\mathfrak{I}_{31}=0,\quad
\mathfrak{I}_{32}=6\a^4\b(b^2\a^2-\b^2)\{5(n-2)b^2\a^2+3\b^2\},\\
&\mathfrak{I}_{33}=-12\a^4\b(b^2\a^2-\b^2)\{(n-8)b^2\a^2+3\b^2\},\quad\mathfrak{I}_{34}=18\a^4\b(b^2\a^2-\b^2)\{(n+2)b^2\a^2-\b^2\},\\
&\mathfrak{I}_{35}=18b^2\a^4\b(b^2\a^2-\b^2)^2,\quad\mathfrak{I}_{36}=18b^2\a^4\b(b^2\a^2-\b^2)^2,\quad
\mathfrak{I}_{37}=-12b^2\a^6\b(b^2\a^2-\b^2).
\end{eqnarray*}

\begin{lemma}
Singular square metric (\ref{singularBerwald}) has scalar Ricci curvature $K(x)$ if and only if
\begin{eqnarray}\label{soiemgienignieng}
\Rat=0,\qquad\Irrat=0.
\end{eqnarray}
\end{lemma}
\begin{proof}
Notice that both $\Rat$ and $\Irrat$ are rational functions of $y$, but $\a$ is irrational. Hence, (\ref{pwpeiming}) holds if and only if (\ref{soiemgienignieng}) holds.
\end{proof}

\begin{lemma}\label{adnaindnignag}
There are two functions $c(x)$, $d(x)$ and a $1$-form $\theta=\theta_i(x)y^i$ with is orthogonal to $\b$~($b^i\theta_i=0$), such that
\begin{eqnarray}\label{eiweimgpeomm}
\roo=c(x)\a^2+d(x)\b^2+2\b\theta.
\end{eqnarray}
\end{lemma}
\begin{proof}
One can verify that
\begin{eqnarray}\label{nangnnad}
b^4(b^2\Rat-\b\Irrat)
&=&(b^2\a^2-\b^2)\cdot\textrm{Poly}_1(y)+(b^4r_{00}-2b^2\b r_0+\b^2r)^2\cdot\textrm{Poly}_2(y)=0,
\end{eqnarray}
where
\begin{eqnarray*}
\textrm{Poly}_1(y)&=&9b^8\a^2(b^2\a^2-\b^2)(b^2\a^2+3\b^2)\Ricoo-3b^6\a^2\big\{2(2n-5)b^4\a^4+(17n-27)b^2\a^2\b^2
\nonumber\\
&&+3(n-1)\b^4\big\}\roho+9b^6\a^2(b^2\a^2-\b^2)\big\{2b^2\a^2+(n-1)\b^2\big\}\soho-72b^8\a^4(b^2\a^2-\b^2)\b\sIohi,\\
\textrm{Poly}_2(y)&=&\a^2\big\{(3n-5)b^4\a^4+(43n-59)b^2\a^2\b^2+2(19n-34)\b^4\big\}.
\end{eqnarray*}

(\ref{nangnnad}) indicates that the polynomial $(b^4r_{00}-2b^2\b r_0+\b^2r)^2\cdot\textrm{Poly}_2(y)$ must be divided exactly by the factor $b^2\a^2-\b^2$. Notice that
\begin{eqnarray*}
\textrm{Poly}_2(y)=\f{1}{b^2}{(b^2\a^2-\b^2)\left\{(3n-5)b^4\a^4+2(23n-32)b^2\a^2\b^2+12(7n-11)\b^4\right\}}+\f{12}{b^2}(7n-11)\b^6,
\end{eqnarray*}
so $\textrm{Poly}_2(y)$ can't be divided exactly by $b^2\a^2-\b^2$. Hence, $(b^4r_{00}-2b^2\b r_0+\b^2r)^2$ must be divided exactly by $b^2\a^2-\b^2$.

When dimension $n\geq3$, as a positive semi-definite quadratic form $b^2\a^2-\b^2$, the rank of its corresponding matrix $(b^2\aij-\bi\bj)$ is $n-1$, which indicates that $b^2\a^2-\b^2$ is irreducible. As a result, $b^4r_{00}-2b^2\b r_0+\b^2r$ must be divided exactly by $b^2\a^2-\b^2$.

When $n=2$, there is a linear function $\gamma\neq0$ such that $b^2\a^2-\b^2=\gamma^2$. As a result, $b^4r_{00}-2b^2\b r_0+\b^2r$ must be divided exactly by $\gamma$. Hence, there is a linear function $\delta$ such that $b^4r_{00}-2b^2\b r_0+\b^2r=\gamma\delta$, i.e.,
\begin{eqnarray}\label{sinwengaaddag}
b^4\rij-b^2(\bi\rj+\bj\ri)+r\bi\bj=\frac{1}{2}(\gamma_i\delta_j+\gamma_j\delta_i).
\end{eqnarray}
By the fact $(b^2\aij-\bi\bj)b^i=0$ we know that $\gamma$ must be orthogonal to $\b$, namely $\gamma_ib^i=0$. Contracting (\ref{sinwengaaddag}) with $b^i$, we can see that $\delta$ is orthogonal to  $\b$ too. Hence, $\gamma$ must parallel to $\delta$ due to the restriction of dimension. So
\begin{eqnarray*}
b^4r_{00}-2b^2\b r_0+\b^2r=\gamma\delta\varpropto\gamma^2=b^2\a^2-\b^2,
\end{eqnarray*}

Summarizing the above arguments, the equality (\ref{nangnnad}) indicates that there is a function $c(x)$, such that
\begin{eqnarray*}
b^4r_{00}-2b^2\b r_0+\b^2r=c(x)b^2(b^2\a^2-\b^2).
\end{eqnarray*}
Further, the above equality indicates that there are a function $d(x)$ and a $1$-form $\theta$, such that
\begin{eqnarray}\label{eiweimgpeomm2}
\roo=c(x)\a^2+d(x)\b^2+2\b\theta.
\end{eqnarray}
Moreover, if $\theta$ is not orthogonal to $\b$, one can replace $\theta$ as
\begin{eqnarray*}
\bar\theta=\theta-\frac{\theta^ib_i}{b^2}\b,
\end{eqnarray*}
in this case, (\ref{eiweimgpeomm2}) reads
\begin{eqnarray*}
\roo=\bar c(x)\a^2+\bar d(x)\b^2+2\b\bar\theta,
\end{eqnarray*}
where
\begin{eqnarray*}
\bar c=c,\qquad\bar d=d+\frac{2\theta^ib_i}{b^2},
\end{eqnarray*}
and meanwhile ${\bar\theta}^ib_i=0$. That is to say, we can always assume that the $1$-form $\theta$ in (\ref{eiweimgpeomm2}) is orthogonal to $\b$.
\end{proof}

\begin{lemma}\label{wingangdlnand}
\begin{eqnarray*}
K=0.
\end{eqnarray*}
\end{lemma}
\begin{proof}
By (\ref{eiweimgpeomm}), we have
\begin{eqnarray}
\ro&=&(c+b^2d)\b+b^2\theta,\nonumber\\
r&=&(c+b^2d)b^2,\nonumber\\
\rIi&=&nc+b^2d,\nonumber\\
\rooho&=&\co\a^2+d_0\b^2+2d\b\roo+2\theta\roo+2\b\theta_{|0},\nonumber\\
\roohb&=&\cb\a^2+\db\b^2+2d\b(\ro-\so)+2\theta(\ro-\so)+2\b\theta_{|b},\nonumber\\
\rIiho&=&n\co+b^2d_0+2d(\ro+\so),\nonumber\\
\rIihb&=&n\cb+b^2\db+2dr,\nonumber\\
\rIohi&=&\co+\db\b+d(\ro-\so)+\theta_{|b}+(d\b+\theta)\rIi+c\theta+\b\theta^i\theta_i-\theta^is_{i0}+\b\theta^i{}_{|i},\nonumber\\
\roho&=&\big\{\co+b^2d_0+2d(\ro+\so)\big\}\b+(c+b^2d)\roo+2\theta(\ro+\so)+b^2\theta_{|0},\nonumber\\
\rohb&=&\big\{\cb+b^2\db+2dr\big\}\b+(c+b^2d)(\ro-\so)+2\theta r+b^2\theta_{|b},\nonumber\\
\rIhi&=&\big\{\cb+b^2\db+2dr\big\}+(c+b^2d)\rIi+2\big(b^2\theta^i\theta_i+\theta^is_i\big)+b^2\theta^i{}_{|i},\nonumber\\
r_{|0}&=&b^2\big\{\co+b^2d_0+2d(\ro+\so)\big\}+2(c+b^2d)(\ro+\so),\nonumber\\
r_{|b}&=&b^2(\cb+b^2\db+2dr)+2(c+b^2d)r,\nonumber\\
\poo&=&c^2\a^2+b^2(d\b+\theta)^2+\b^2\theta^i\theta_i+2c(d\b+\theta)\b+2c\b\theta,\nonumber\\
\po&=&c(c+b^2d)\b+b^2c\theta+b^2(c+b^2d)(d\b+\theta)+b^2\b\theta^i\theta_i,\nonumber\\
p&=&b^2(c+b^2d)^2+b^4\theta^i\theta_i,\nonumber\\
\pIi&=&nc^2+b^2(b^2d^2+\theta^i\theta_i)+b^2\theta^i\theta_i+2cb^2d,\nonumber\\
\qoo&=&(d\b+\theta)\so+\b\theta^is_{i0},\nonumber\\
\qqo&=&-c\so-\b\theta^is_i,\nonumber\\
\qo&=&(c+b^2d)\so+b^2\theta^is_{i0},\nonumber\\
q&=&-b^2\theta^is_i.\label{downieng}
\end{eqnarray}
Plugging (\ref{eiweimgpeomm}) and (\ref{downieng}) into (\ref{IrratBerwald}) yields
\begin{eqnarray}
\Irrat=\a^2\cdot\textrm{Poly}_3(y)-18(n-1)Kb^4\b^9=0,\label{cniangengak}
\end{eqnarray}
where
\begin{eqnarray}
\textrm{Poly}_3(y)&=&-18(n-1)Kb^6(b^2\a^2-2\b^2)(b^4\a^4-2b^2\a^2\b^2+2\b^4)\b-18b^4(b^2\a^2-\b^2)^2\b\Ricoo
-9(n-1)\nonumber\\
&&\cdot(b^2\a^2-\b^2)^3\b c^2+9(n-1)b^2(b^2\a^2-\b^2)^2(b^2\a^2+\b^2)\b cd-2(n-3)b^4(b^2\a^2-\b^2)^3\b d^2\nonumber\\
&&-9(n-1)b^2(b^2\a^2-\b^2)^2(b^2\a^2+\b^2)\co-6b^4(b^2\a^2-\b^2)^2\big\{(n-1)b^2\a^2-(n-3)\b^2\big\}d_{0}\nonumber\\
&&+12b^4\a^2(b^2\a^2-\b^2)^2\b\db-18(b^2\a^2-\b^2)^2\big\{(2n-3)b^2\a^2-(n-1)\b^2\big\}c\so-6b^2(b^2\a^2-\b^2)^2\nonumber\\
&&\cdot\big\{2(n+1)b^2\a^2-3(n-1)\b^2\big\}d\so+9\big\{(9n-17)b^4\a^4+20b^2\a^2\b^2+(n-1)\b^4\big\}\b\so^2\nonumber\\
&&+36b^4\a^2(b^2\a^2-\b^2)^2\sIohi-9(n-1)b^2(b^2\a^2-\b^2)^2\b\soho-18b^2\a^2(b^2\a^2-\b^2)^2\sohb\nonumber\\
&&+18b^2\a^2(b^2\a^2-\b^2)^2\b\sIhi+72b^4\a^2(b^2\a^2-\b^2)\b\too-18b^2\a^2(b^2\a^2-\b^2)\big\{11b^2\a^2\nonumber\\
&&-(2n+9)\b^2\big\}\to+18\a^2(b^2\a^2-\b^2)\big\{(n+2)b^2\a^2-\b^2\big\}\b t+6b^2(b^2\a^2-\b^2)^2\big\{(3n-4)b^2\a^2\nonumber\\
&&+3(n-1)\b^2\big\}c\theta-4b^4(b^2\a^2-\b^2)^2\big\{2b^2\a^2+(n-3)\b^2\big\}d\theta+12b^2\big\{(7n-15)b^4\a^4\nonumber\\
&&-(5n-19)b^2\a^2\b^2
+3(n-1)\b^4\big\}\b\so\theta-b^4\big\{(17n-33)b^4\a^4-4(13n-22)b^2\a^2\b^2\nonumber\\
&&+(25n-57)\b^4\big\}\b\theta^2
+6b^4\a^2(b^2\a^2-\b^2)\big\{(8n-37)b^2\a^2-(6n-43)\b^2\big\}\theta^is_{i0}-6b^2\a^2(b^2\a^2-\b^2)\nonumber\\
&&\cdot\big\{5b^2\a^2+(4n-5)\b^2\big\}\b\theta^is_i
+6b^4\a^2(b^2a^2-\b^2)\big\{3(n-1)b^2\a^2-2(2n-1)\b^2\big\}\b\theta^i\theta_i\nonumber\\
&&+3(7n-15)b^4(b^2\a^2-\b^2)^2\b\theta_{0|0}+6b^4\a^2(b^2\a^2-\b^2)^2\theta_{0|b}
+18b^4\a^2(b^2\a^2-\b^2)^2\b\theta^i{}_{|i}.\label{cminaing}
\end{eqnarray}
Since $\b^9$ can't be divided by $\a^2$ exactly, we know $K=0$.
\end{proof}

Lemma \ref{wingangdlnand} shows that any Einstein singular square metric must be Ricci flat. This conclusion coincides with that of regular square metrics.

\begin{lemma}
When $n\geq3$,
\begin{eqnarray}
\theta=-\f{3}{b^2}\so\label{uenanngel}
\end{eqnarray}
\end{lemma}
\begin{proof}
Plugging $K=0$ into (\ref{cniangengak}) and combining with (\ref{cminaing}) yields
\begin{eqnarray}
\frac{\Irrat}{\a^2}=(b^2\a^2-\b^2)\cdot\textrm{Poly}_4(y)+2(5n+1)\b^5(b^2\theta+3\so)^2=0.\label{baidbgba}
\end{eqnarray}
where
\begin{eqnarray}
\textrm{Poly}_4(y)&=&-18b^4(b^2\a^2-\b^2)\b\Ricoo-9(n-1)(b^2\a^2-\b^2)^2\b c^2+9(n-1)b^2(b^2\a^2-\b^2)(b^2\a^2+\b^2)\b cd\nonumber\\
&&-2(n-3)b^4(b^2\a^2-\b^2)^2\b d^2-9(n-1)b^2(b^2\a^2-\b^2)(b^2\a^2+\b^2)\co-6b^4(b^2\a^2-\b^2)\nonumber\\
&&\cdot\big\{(n-1)b^2\a^2-(n-3)\b^2\big\}d_{0}+12b^4\a^2(b^2\a^2-\b^2)\b\db-18(b^2\a^2-\b^2)\big\{(2n-3)b^2\a^2\nonumber\\
&&-(n-1)\b^2\big\}c\so-6b^2(b^2\a^2-\b^2)\big\{2(n+1)b^2\a^2-3(n-1)\b^2\big\}d\so+9\b\big\{(9n-17)b^2\a^2\nonumber\\
&&+3(3n+1)\b^2\big\}\so^2+36b^4\a^2(b^2\a^2-\b^2)\sIohi-9(n-1)b^2(b^2\a^2-\b^2)\b\soho-18b^2\a^2(b^2\a^2-\b^2)\nonumber\\
&&\cdot\sohb+18b^2\a^2(b^2\a^2-\b^2)\b\sIhi+72b^4\a^2\b\too-18b^2\a^2\big\{11b^2\a^2-(2n+9)\b^2\big\}\to
+18\a^2\b\big\{(n+2)\nonumber\\
&&b^2\a^2-\b^2\big\}t+6b^2(b^2\a^2-\b^2)\big\{(3n-4)b^2\a^2+3(n-1)\b^2\big\}c\theta-4b^4(b^2\a^2-\b^2)\big\{2b^2\a^2
+(n-3)\nonumber\\
&&\cdot\b^2\big\}d\theta+12b^2\b\big\{(7n-15)b^2\a^2+2(n+2)\b^2\big\}\so\theta-b^4\big\{(17n-33)b^2\a^2
-5(7n-11)\b^2\big\}\b\theta^2\nonumber\\
&&+6b^4\a^2\big\{(8n-37)b^2\a^2-(6n-43)\b^2\big\}\theta^is_{i0}-6b^2\a^2\big\{5b^2\a^2+(4n-5)\b^2\big\}\b\theta^is_i
+6b^4\a^2\nonumber\\
&&\cdot\big\{3(n-1)b^2\a^2-2(2n-1)\b^2\big\}\b\theta^i\theta_i+3(7n-15)b^4(b^2\a^2-\b^2)\b\theta_{0|0}
+6b^4\a^2(b^2\a^2-\b^2)\theta_{0|b}\nonumber\\
&&+18b^4\a^2(b^2\a^2-\b^2)\b\theta^i{}_{|i}.\label{onaienkgnkg}
\end{eqnarray}

Due to (\ref{baidbgba}), we know that $(b^2\theta+3\so)^2$ must be divided by $b^2\a^2-\b^2$ exactly since $\b^5$ can't be. However, $(b^2\theta+3\so)^2$ is reducible but $b^2\a^2-\b^2$ is irreducible  when $n\geq3$, which have been pointed out in the proof of Lemma \ref{adnaindnignag}. Hence, such divisibility relationship doesn't exist unless $b^2\theta+3\so=0$.
\end{proof}

\begin{lemma}
When $n\geq3$, the functions $c(x)$ and $d(x)$ must satisfy condition
\begin{eqnarray}\label{pawmng}
9b^2\co+6b^4d_0-(9c^2+9b^2cd-2b^4d^2)\b+6(9c+2b^2d)\so=0,
\end{eqnarray}
\end{lemma}
\begin{proof}
By (\ref{uenanngel}) we have
\begin{eqnarray}\label{oweningieniang}
&\displaystyle\theta^is_{i0}=-\f{3}{b^2}\to,\qquad\theta^i\si=\f{3}{b^2}t,\qquad\theta^i\theta_i=-\f{9}{b^4}t,&\nonumber\\
&\displaystyle\theta_{|0}=-\f{3}{b^2}\soho+\f{6}{b^4}(\ro+\so)\so,\qquad\theta_{|b}=-\f{3}{b^2}\sohb+\f{6}{b^4}\so r,\qquad\theta^i{}_{|i}=-\f{3}{b^2}\sIhi-\f{6}{b^4}(q+t).&
\end{eqnarray}
On the other hand, by the priori formula (\ref{relation5}) and $\rohb$ in (\ref{downieng}), we have
\begin{eqnarray}\label{wminibeubgodn}
\sohb=-\f{b^4(\co+b^2d_0)-b^2\b(\cb+b^2\db)+2b^2(2c-b^2d)\so+2b^2\to+6\b t}{2b^2}.
\end{eqnarray}

Plugging (\ref{uenanngel}), (\ref{oweningieniang}) and (\ref{wminibeubgodn}) into (\ref{baidbgba}), and combining with (\ref{onaienkgnkg}), yields
\begin{eqnarray}
\frac{b^2\Irrat}{\a^2(b^2\a^2-\b^2)}
=(b^2\a^2-\b^2)\cdot\textrm{Poly}_5(y)+72\b^3(b^4\too-2b^2\b\to+\b^2t+b^2\so^2)=0,\label{imimangng}
\end{eqnarray}
where
\begin{eqnarray}
\textrm{Poly}_5(y)&=&-18b^6\b\Ricoo-9(n-1)b^2(b^2\a^2-\b^2)\b c^2+9(n-1)b^4(b^2\a^2+\b^2)\b cd-2(n-3)b^6(b^2\a^2
-\b^2)\nonumber\\
&&\cdot\b d^2-9b^4\big\{(n-3)b^2\a^2+(n-1)\b^2\big\}\co-6b^6\big\{(n-4)b^2\a^2-(n-3)\b^2\big\}d_{0}-18b^4\a^2\b\cb\nonumber\\
&&-6b^6\a^2\b\db-18b^2(5n-13)(b^2\a^2-\b^2)c\so-12b^4\big\{(n-1)b^2\a^2-(13n-27)\b^2\big\}d\so\nonumber\\
&&-72(8n-17)b^2\b\so^2+72b^4\b\too-72b^2\big\{(2n-7)b^2\a^2+2\b^2\big\}\to-72\big\{2(n-3)b^2\a^2-\b^2\big\}\b t\nonumber\\
&&+36b^6\a^2\sIohi-36b^4\a^2\b\sIhi-72(n-2)b^4\b\soho.\label{hdonangga}
\end{eqnarray}
(\ref{imimangng}) indicates that $b^4\too-2b^2\b\to+\b^2t+b^2\so^2$ can be divided by $b^2\a^2-\b^2$ exactly. Hence, there is a function $\sigma(x)$ such that
\begin{eqnarray}\label{iweigmirabdng}
\too=\sigma(x)(b^2\a^2-\b^2)+\f{2}{b^2}\b\to-\f{1}{b^4}\b^2t-\f{1}{b^2}\so^2.
\end{eqnarray}
As a result, we have
\begin{eqnarray}\label{doweminigniang}
\tIi=(n-1)\sigma b^2+\f{2}{b^2}t.
\end{eqnarray}

Plugging (\ref{iweigmirabdng}) and (\ref{doweminigniang}) into (\ref{imimangng}), and combining with (\ref{hdonangga}), $\Irrat$ reads \begin{eqnarray}
\Irrat&=&\a^2(b^2\a^2-\b^2)^2\Big\{-18b^4\b\Ricoo+72b^4\a^2\b\sigma-9(n-1)(b^2\a^2-\b^2)\b c^2
+9(n-1)b^2(b^2\a^2+\b^2)\b cd\nonumber\\
&&-2(n-3)b^4(b^2\a^2-\b^2)\b d^2-9b^2[(n-3)b^2\a^2+(n-1)\b^2]\co-6b^4[(n-4)b^2\a^2-(n-3)\b^2]d_{0}\nonumber\\
&&-6b^2\a^2\b(3\cb+b^2\db)-18(5n-13)(b^2\a^2-\b^2)c\so-12b^2[(n-1)b^2\a^2-(13n-27)\b^2]d\so\nonumber\\
&&-576(n-2)\b\so^2-72(2n-7)b^2\a^2\to-144(n-3)\a^2\b t+36b^4\a^2\sIohi-36b^2\a^2\b\sIhi\nonumber\\
&&-72(n-2)b^2\b\soho\Big\}.\label{naengnkdd}
\end{eqnarray}
Meanwhile, plugging all the related formulae into (\ref{RatBerwald}), $\Rat$ reads
\begin{eqnarray}
\Rat&=&\f{1}{b^2}\a^2(b^2\a^2-\b^2)^2\Big\{9b^4(b^2\a^2+\b^2)\Ricoo-36b^4\a^2(nb^2\a^2-3\b^2)\sigma
+9(n-1)b^2\a^2(b^2\a^2-\b^2)c^2\nonumber\\
&&-18(n-1)b^4\a^2\b^2cd-2b^4(b^2\a^2-\b^2)[b^2\a^2-(n-2)\b^2]d^2+18(n-2)b^4\a^2\b\co+6b^4[(n-3)b^2\a^2\nonumber\\
&&-(n-2)\b^2]\b d_{0}+18b^2\a^2\b^2\cb-6b^4\a^2(b^2\a^2-2\b^2)\db+72(n-2)(b^2\a^2-\b^2)\b c\so\nonumber\\
&&-12b^2[(5n-12)b^2\a^2+7(n-2)\b^2]\b d\so+288(n-2)(b^2\a^2+\b^2)\so^2+72(2n-7)b^2\a^2\b\to\nonumber\\
&&+144(n-3)b^2\a^4t-36b^4\a^2\b\sIohi+36b^4\a^4\sIhi+36(n-2)b^2(b^2\a^2+\b^2)\soho\Big\}.\label{naengnkd}
\end{eqnarray}

By (\ref{naengnkdd}) and (\ref{naengnkd}) we have
\begin{eqnarray}
\frac{2b^2\b\Rat+(b^2\a^2+\b^2)\Irrat}{\a^2(b^2\a^2-\b^2)^2}
&=&\a^2\cdot\textrm{Poly}_6(y)-(n-1)\b^4\big\{9b^2\co+6b^4d_0\nonumber\\
&&-(9c^2+9b^2cd-2b^4d^2)\b+6(9c+2b^2d)\so\big\}\nonumber\\
&=&0,\label{owmiengka}
\end{eqnarray}
where
\begin{eqnarray}
\textrm{Poly}_6(y)&=&-72b^4\b\big\{(n-1)b^2\a^2-4\b^2\big\}\sigma+(n-1)b^2(b^2\a^2-2\b^2)\b(9c^2+9b^2cd-2b^4d^2)
-9b^4\big\{(n-3)b^2\a^2\nonumber\\
&&-2(n-2)\b^2\big\}\co-6b^6\big\{(n-4)b^2\a^2-(2n-5)\b^2\big\}d_{0}-18b^2\b(b^2\a^2-\b^2)(\cb+b^2\db)\nonumber\\
&&-18b^2\big\{(5n-13)b^2\a^2-8(n-2)\b^2\big\}c\so-12(n-1)b^4(b^2\a^2-2\b^2)d\so-72(2n-7)b^2\nonumber\\
&&\cdot(b^2\a^2-\b^2)\to+144(n-3)(b^2\a^2-\b^2)\b t+36b^4(b^2\a^2-\b^2)\sIohi+36b^2(b^2\a^2-\b^2)\b\sIhi.\label{ganndnga}
\end{eqnarray}
(\ref{owmiengka}) indicates that $\b^4\big\{9b^2\co+6b^4d_0-(9c^2+9b^2cd-2b^4d^2)\b+6(9c+2b^2d)\so\big\}$ must be divided by $\a^2$ exactly. Notice that both $\b$ and $9b^2\co+6b^4d_0-(9c^2+9b^2cd-2b^4d^2)\b+6(9c+2b^2d)\so$ are linear functions of $y$. Hence such divisibility relationship doesn't exist unless (\ref{pawmng}) holds.
\end{proof}

\begin{lemma}
When $n\geq3$,
\begin{eqnarray}\label{omdingadnnd}
b^2s_{i0}+\b\si-\so\bi=0.
\end{eqnarray}
In particular,
\begin{eqnarray}\label{fwoeningaidng}
\too=\f{1}{b^4}\b^2t-\f{1}{b^2}\so^2,\qquad\to=\f{1}{b^2}\b t,\qquad\tIi=\f{2}{b^2}t.
\end{eqnarray}
\end{lemma}
\begin{proof}
By (\ref{pawmng}) we have
\begin{eqnarray}
\co=-\f{6b^4d_0-(9c^2+9b^2cd-2b^4d^2)\b+6(9c+2b^2d)\so}{9b^2},\label{cosidndaignd}
\end{eqnarray}
hence,
\begin{eqnarray}
\cb=-\f{6b^2\db-(9c^2+9b^2cd-2b^4d^2)}{9}.\label{cbsidndaignd}
\end{eqnarray}
Plugging (\ref{cosidndaignd}) and (\ref{cbsidndaignd}) into (\ref{owmiengka}), and combining with (\ref{ganndnga}), we have
\begin{eqnarray}
\frac{2b^2\b\Rat+(b^2\a^2+\b^2)\Irrat}{6\a^4(b^2\a^2-\b^2)^2}
&=&b^2\a^2\cdot\textrm{Poly}_7(y)-\b^2\big\{6b^4\sIohi+6b^2\b\sIhi+b^6d_0-b^4\b\db\nonumber\\
&&-6(n-2)b^2c\so-4b^4d\so-12(2n-7)b^2\to+24(n-3)\b t-48\sigma b^4\b\big\}\nonumber\\
&=&0,\label{dgnandnga}
\end{eqnarray}
where
\begin{eqnarray}
\textrm{Poly}_7(y)&=&-12(n-1)b^4\b\sigma+b^6d_{0}-b^4\b\db-6(n-2)b^2c\so-4b^4d\so-12(2n-7)b^2\to+24(n-3)\b t\nonumber\\
&&+6b^4\sIohi+6b^2\b\sIhi.\label{dgnandnga2}
\end{eqnarray}
So we have
\begin{eqnarray}
&&6b^4\sIohi+6b^2\b\sIhi+b^6d_0-b^4\b\db-6(n-2)b^2c\so\nonumber\\
&&-4b^4d\so-12(2n-7)b^2\to+24(n-3)\b t-48\sigma b^4\b=0.\label{ieimiinwineing}
\end{eqnarray}
Differentiating the above equality with respect to $y^k$ and then contracting with $b^k$ yields
\begin{eqnarray*}
-6(n+7)\sigma b^6=0,
\end{eqnarray*}
Notice that we use the facts $\sIkhi b^k=-\sIhi-\qIi-\tIi$ and $\qIi=0$ here. As a result, we conclude that
\begin{eqnarray}\label{sigmazero}
\sigma=0.
\end{eqnarray}
In this case, (\ref{iweigmirabdng}) reads
\begin{eqnarray*}
b^4\too-2b^2\b\to+\b^2t+b^2\so^2=0.
\end{eqnarray*}
which is equivalent to
\begin{eqnarray*}
(b^2s_{i0}+\b\si-\so\bi)a^{ij}(b^2s_{j0}+\b\sj-\so\bj)=0.
\end{eqnarray*}
Regard $b^2s_{i0}+\b\si-\so\bi$ as a vector, then the above equality means that the norm of such vector is zero. Hence, (\ref{omdingadnnd}) holds.
\end{proof}

As a summary, we have the following conclusion.
\begin{proposition}
When $n\geq3$, the $1$-form of an Einstein singular square metric (\ref{singularBerwald}) must satisfy
\begin{eqnarray*}
\rij=c(x)\aij+d(x)\bi\bj-\f{3}{b^2}(\bi\sj+\bj\si),\qquad\sij=\f{1}{b^2}(\bi\sj-\bj\si),
\end{eqnarray*}
where $c(x)$ and $d(x)$ must satisfy three additional conditions
\begin{eqnarray}
&9b^2\co+6b^4d_0-(9c^2+9b^2cd-2b^4d^2)\b+6(9c+2b^2d)\so=0,&\label{woeminaddd'amd}\\
&b^4(\co+b^2d_0)-b^2\b(\cb+b^2\db)+2b^2(2c-b^2d)\so+2b^2\to+6\b t+2b^2\sohb=0,&\label{wminibeubgodn2}\\
&b^6d_0-b^4\b\db-6(n-2)b^2c\so-4b^4d\so+12\b t+6b^2\b\sIhi+6b^4\sIohi=0.&\label{siohiBerwald}
\end{eqnarray}
\end{proposition}
\begin{proof}
Contracting (\ref{omdingadnnd}) with $s^i$ yields
\begin{eqnarray}
\to=\f{1}{b^2}\b t.\label{tdnguangknd}
\end{eqnarray}
Plugging (\ref{sigmazero}) and (\ref{tdnguangknd}) into (\ref{ieimiinwineing}), we can obtain (\ref{siohiBerwald}).
\end{proof}

Now we can proof the main result of this paper.
\begin{proof}[Proof of Theorem \ref{iwneingadad}]
Necessity:~Combining (\ref{dgnandnga}), (\ref{dgnandnga2}) and (\ref{siohiBerwald}) one can see that $\Rat$ is a scaling of $\Irrat$. hence, we just need to consider $\Irrat$ from now on.

Plugging (\ref{cosidndaignd}), (\ref{cbsidndaignd}), (\ref{sigmazero}), (\ref{tdnguangknd}) and (\ref{siohiBerwald}) into (\ref{naengnkdd}), $\Irrat$ reads
\begin{eqnarray*}
\Irrat&=&-2\a^2(b^2\a^2-\b^2)^2\b\big\{9b^4\Ricoo+[9(n-1)b^2c^2-2b^6d^2-6b^4\db+144(n-3)t+36 b^2\sIhi]\a^2\\
&&-2(n-2)b^4d^2\b^2-6(n-2)b^4\b d_0-72(n-2)c\b\so-84(n-2)b^2d\b\so+288(n-2)\so^2\\
&&+36(n-2)b^2\soho\big\}\\
&=&0.
\end{eqnarray*}
Hence, $\Ricoo$ can be solved as below,
\begin{eqnarray*}
\Ricoo&=&-\f{1}{9b^4}\big\{[9(n-1)b^2c^2-2b^6d^2-6b^4\db+144(n-3)t+36 b^2\sIhi]\a^2-2(n-2)b^4d^2\b^2\\
&&-6(n-2)b^4\b d_0-72(n-2)c\b\so-84(n-2)b^2d\b\so+288(n-2)\so^2+36(n-2)b^2\soho\big\},
\end{eqnarray*}
then we have
\begin{eqnarray*}
\Rico&=&-\f{1}{9b^4}\big\{[9(n-1)b^2c^2-2(n-1)b^6d^2-3nb^4\db+18(5n-18)t+36b^2\sIhi]\b\nonumber\\
&&-3(n-2)\left[2b^6d_0-b^4\b\db+4b^2(3c+b^2d)\so+30\b t\right]\big\},
\end{eqnarray*}
and
\begin{eqnarray*}
\Ric&=&-\f{1}{9b^2}\big\{9(n-1)b^2c^2-2(n-1)b^6d^2-6(n-1)b^4\db-144t+36b^2\sIhi\big\}.
\end{eqnarray*}
One can verify that the priori formulae (\ref{relation2}) and (\ref{relation6}) holds automatically.

Sufficiency:~According to the whole discussions before, the conditions (\ref{bewaldcondiononalpha})--(\ref{bewaldcondiononbetas}) and (\ref{woeminaddd'amd})--(\ref{siohiBerwald}) are sufficient to make a singular square metric (\ref{singularBerwald}) to be a Ricci-flat Einstein metric. However, (\ref{woeminaddd'amd})--(\ref{siohiBerwald}) can be rebuilt by (\ref{bewaldcondiononalpha})--(\ref{bewaldcondiononbetas}) combining with the priori formulae.

First, by (\ref{bewaldcondiononalpha}) we have
\begin{eqnarray*}
\Rico&=&-\f{1}{9b^4}\big\{[9(n-1)b^2c^2-2(n-1)b^6d^2-3nb^4\db+18(5n-18)t+36b^2\sIhi]\b-3(n-2)\big[2b^6d_0-b^4\b\db\nonumber\\
&&+4b^2(3c+b^2d)\so+30\b t\big]\big\}+\f{n-2}{9b^4}\left\{9X-b^2\b Y+b^2Z\right\},\\
\Ric&=&-\f{1}{9b^2}\big\{9(n-1)b^2c^2-2(n-1)b^6d^2-6(n-1)b^4\db-144t+36b^2\sIhi\big\},
\end{eqnarray*}
and by (\ref{bewaldcondiononbeta})--(\ref{bewaldcondiononbetas}) we have
\begin{eqnarray}\label{dinagninaaduaadgg}
\sIohi=-\f{1}{6b^4}\big\{b^6d_0-b^4\b\db-6(n-2)b^2c\so-4b^4d\so+12\b t+6b^2\b\sIhi\big\}-\f{1}{18b^4}\left\{9X-b^2\b Y+b^2Z\right\},~~~~
\end{eqnarray}
where
\begin{eqnarray*}
X&=&-2b^2\sohb-b^4(\co+b^2d_0)+b^2\b(\cb+b^2\db)-2b^2(2c-b^2d)\so-2b^2\to-6\b t,\\
Y&=&9\cb+6b^2\db-(9c^2+9b^2cd-2b^4d^2),\\
Z&=&9b^2\co+6b^4d_0-(9c^2+9b^2cd-2b^4d^2)\b+6(9c+2b^2d)\so.
\end{eqnarray*}

Now, it is easy to verify that the priori formulae (\ref{relation5}), (\ref{relation2}) and (\ref{relation6}) read
\begin{eqnarray*}
\sohb+p_{0}+q_{0}-q^{*}_{0}+t_{0}+\rohb-r_{|0}&=&\frac{X}{b^2},\\
\Ric+\pIi+\tIi+\rIihb-\rIhi+\sIhi&=&\f{(n-1)Y}{9},\\
\Rico+\rIiho-\rIohi-\sIohi&=&\f{9(n-3)X-(n-3)b^2\b Y+2(n-2)b^2Z}{9b^4}.
\end{eqnarray*}
So, when $n\geq3$, the priori formulae ask $X=Y=Z=0$. Obviously, $X=0$ and $Z=0$ are equivalent to (\ref{wminibeubgodn2}) and (\ref{woeminaddd'amd}) respectively, and $Y=0$ can be obtained by $Z=0$. Finally, (\ref{siohiBerwald}) holds according to (\ref{dinagninaaduaadgg}). Hence, (\ref{bewaldcondiononalpha})--(\ref{bewaldcondiononbetas}) are sufficient to make a singular square metric (\ref{singularBerwald}) be a Ricci-flat Einstein metric.
\end{proof}

\section{Constant flag curvature:~necessary and sufficient conditions}
\begin{proof}[Proof of Theorem \ref{odnanganndnand}]
Necessity:~Since $F$ is an Einstein metric with Ricci constant $K$, according to Theorem \ref{iwneingadad}, $K=0$, (\ref{bewaldcondiononbeta})--(\ref{bewaldcondiononbetas}) hold naturally.

(\ref{bewaldcondiononbeta}) is equivalent to (\ref{bewaldcondiononbetarijflag}), by which we can obtain some basic facts:
\begin{eqnarray*}
\rIk&=&c\dIk+d\bI\bk+(\bI\theta_k+\theta^i\bk),\nonumber\\
\rIkho&=&\co\dIk+d_0\bI\bk+d\bI(\rko+\sko)+d(\rIo+\sIo)\bk+\theta^i(\rko+\sko)+(\rIo+\sIo)\theta_k+\bI\theta_{k|0}+\theta^i{}_{|0}\bk,\nonumber\\
\rIohk&=&\yI\ck+\bI\{\b\dk+d(\rko-\sko)+\theta_{0|k}\}+(d\b+\theta)(\rIk+\sIk)+\theta^i(\rko-\sko)+\b\theta^i{}_{|k},\nonumber\\
\rIhk&=&\bI\{\ck+b^2\dk+2d(\rk+\sk)\}+(c+b^2d)(\rIk+\sIk)+2\theta^i(\rk+\sk)+b^2\theta^i{}_{|k},\nonumber\\
\rhk&=&b^2\{\ck+b^2\dk+2d(\rk+\sk)\}+2(c+b^2d)(\rk+\sk),\nonumber\\
\pIk&=&c^2\dIk+(2cd+b^2d^2+\theta^l\theta_l)\bI\bk+(2c+b^2d)(\bI\theta_k+\theta^i\bk)+b^2\theta^i\theta_k,\nonumber\\
\pk&=&(c+b^2d)^2\bk+(2c+b^2d)b^2\theta_k+b^2\theta^l\theta_l\bk,\nonumber\\
\qko&=&c\sko+(d\so+\theta^ls_{l0})\bk+\so\theta_k,\nonumber\\
\qok&=&-c\sko+(d\b+\theta)\sk+\b\theta^ls_{lk},\nonumber\\
\qk&=&(c+b^2d)\sk+b^2\theta^ls_{lk},\nonumber\\
\qqk&=&-c\sk-\theta^ls_l\bk,
\end{eqnarray*}
and the other related quantities $\rIo$, $\rko$, $\rIkhb$, $\rIoho$, $\rkoho$, $\rIohb$, $\rkohb$, $\rkohI$, $\roohI$, $\roohk$, $\rIho$, $\rkho$, $\rohk$, $\rIhb$, $\rkhb$, $\rkhI$, $\rohI$, $\rhI$, $\pIo$, $\pko$, $\pI$, $\qIo$, $\qoI$, $\qI$ and $\qqI$.

(\ref{woeminaddd'amd}) yields
\begin{eqnarray*}
\co=-\f{1}{9b^2}\left\{6b^4d_0-(9c^2+9b^2 cd-2b^4 d^2)\b+6(9c+2b^2d)\so\right\}.
\end{eqnarray*}
As a result, we can obtain $\cI$, $\ck$, and
\begin{eqnarray*}
\cb=-\f{1}{9}\left\{6b^2\db-(9c^2+9b^2 cd-2b^4 d^2)\right\}.
\end{eqnarray*}

(\ref{bewaldcondiononbetas}) is equivalent to (\ref{bewaldcondiononbetasijflag}), by which we can solve $\sIk$~ºÍ~$\sko$. Meanwhile, by (\ref{fwoeningaidng}) we obtain
\begin{eqnarray*}
\tIk=\f{t\bI\bk-b^2\sI\sk}{b^4},\qquad\tk=\f{t\bk}{b^2},
\end{eqnarray*}
and the other related quantities $\tIo$, $\tko$ and $\tI$.

Differentiating $b^2\sIk-\bI\sk+\sI\bk=0$ yields
\begin{eqnarray*}
0&=&2(\ro+\so)\sIk+b^2\sIkho-(\rIo+\sIo)\sk-\bI\skho+\bk\sIho+\sI(\rko+\sko)\\
&=&-\f{1}{b^2}\left\{b^2c(\yI\sk-\yk\sI)-[(2c+b^2d)\b-2\so](\bI\sk-\bk\sI)-b^4\sIkho+b^2(\bI\skho-\bk\sIho)\right\},
\end{eqnarray*}
by which we can obtain
\begin{eqnarray}
\sIkho&=&\f{1}{b^4}\left\{b^2c(\yI\sk-\yk\sI)-[(2c+b^2d)\b-2\so](\bI\sk-\bk\sI)+b^2(\bI\skho-\bk\sIho)\right\},\nonumber\\
\sIohk&=&\f{1}{b^4}\left\{b^2c(\dIk\so-\yk\sI)-[(2c+b^2d)\bk-2\sk](\bI\so-\b\sI)+b^2(\bI\sohk-\b\sIhk)\right\},
\label{singularBerwaldsiohk}\\
\sIoho&=&\f{1}{b^4}\left\{b^2c(\yI\so-\a^2\sI)-[(2c+b^2d)\b-2\so](\bI\so-\b\sI)+b^2(\bI\soho-\b\sIho)\right\}.\nonumber
\end{eqnarray}
Moreover, (\ref{3ingaindngiad}) is equivalent to
\begin{eqnarray}\label{yangenga}
b^4(\bI\dk-\bk\dI)-2(6c+5b^2d)(\bI\sk-\bk\sI)-6b^2(\sIhk-\skhI)=0,
\end{eqnarray}
by which we can obtain
\begin{eqnarray}\label{singulareBerwaldsihk}
\sohk=-\f{1}{6b^2}\left\{[b^4d_0-2(6c+5b^2d)\so]\bk-b^4\b\dk+2(6c+5b^2d)\b\sk-6b^2\skho\right\}
\end{eqnarray}
and $\sohI$.

Plugging all the relationship above into (\ref{isnnanudngad}) yields
\begin{eqnarray*}
\RIkF&=&\RIk+\f{1}{9b^4}\Big\{[(9b^2c^2+144t)\a^2-2b^4d^2\b^2-6(b^4d_0+12c\so+14b^2d\so)\b+288\so^2\nonumber\\
&&+36b^2\soho]\dIk-(9b^2c^2+144t)\yI\yk+2b^2(b^2d^2\b+3b^2d_0+6d\so)(\yI\bk+\yk\bI)\nonumber\\
&&-2b^4d^2\a^2\bI\bk-6b^4\a^2\bI\dk+72[(c+b^2d)\b-4\so](\yI\sk+\yk\sI)-12b^2d\a^2\bI\sk\nonumber\\
&&-72(c+b^2d)\a^2\bk\sI+288\a^2\sI\sk-36b^2(\yI\skho+\yk\sIho)+36b^2\a^2\sIhk\Big\}.
\end{eqnarray*}
So by the equation $\RIkF=0$ we can solve
\begin{eqnarray*}
\RIk&=&-\f{1}{9b^4}\Big\{[(9b^2c^2+144t)\a^2-2b^4d^2\b^2-6(b^4d_0+12c\so+14b^2d\so)\b+288\so^2\nonumber\\
&&+36b^2\soho]\dIk-(9b^2c^2+144t)\yI\yk+2b^2(b^2d^2\b+3b^2d_0+6d\so)(\yI\bk+\yk\bI)\nonumber\\
&&-2b^4d^2\a^2\bI\bk-6b^4\a^2\bI\dk+72[(c+b^2d)\b-4\so](\yI\sk+\yk\sI)-12b^2d\a^2\bI\sk\nonumber\\
&&-72(c+b^2d)\a^2\bk\sI+288\a^2\sI\sk-36b^2(\yI\skho+\yk\sIho)+36b^2\a^2\sIhk\Big\}.
\end{eqnarray*}
Notice that the Riemann tensor $R_{ik}$ is symmetric about two lower index, so by the above equality combining with (\ref{yangenga}) leads to (\ref{bewaldcondiononalphaflag}). Now, it can be verified that the priori formulae (\ref{relation7})--(\ref{relation4}) hold automatically.

Sufficiency:~According to the whole discussions before, the conditions (\ref{bewaldcondiononalphaflag})--(\ref{bewaldcondiononbetasijflag}), (\ref{woeminaddd'amd}), (\ref{singularBerwaldsiohk}) and (\ref{singulareBerwaldsihk}) are sufficient to make a singular square metric (\ref{singularBerwald}) to be a Finsler metric with vanishing flag curvature. However, (\ref{bewaldcondiononalphaflag}) yields (\ref{bewaldcondiononalpha}), so (\ref{woeminaddd'amd}) holds if (\ref{bewaldcondiononalphaflag})--(\ref{bewaldcondiononbetasijflag}) hold according the proof of Theorem \ref{iwneingadad}. On the other hands, (\ref{singularBerwaldsiohk}) and (\ref{singulareBerwaldsihk}) can be rebuilt by (\ref{bewaldcondiononalphaflag})--(\ref{bewaldcondiononbetasijflag}) combining with the priori formulae (\ref{relation1}) and (\ref{3ingaindngiad}) respectively.
\end{proof}

\noindent Changtao Yu\\
School of Mathematical Sciences, South China Normal
University, Guangzhou, 510631, P.R. China\\
aizhenli@gmail.com

\clearpage
\setcounter{section}{5}
\setcounter{equation}{0}
\section*{Appendix:~Riemann curvature of singular square metrics}
\begin{proposition}
The Riemann tensor of singular square metrics (\ref{singularBerwald}) are given by
\begin{eqnarray}\label{isnnanudngad}
\RIkF&=&\RIk+\big\{c_{101}\roo^2+c_{102}\roo\ro+c_{103}\roo\so+c_{104}\roo r+c_{105}\rooho+c_{106}\ro^2+c_{107}\ro\so+c_{108}\so^2+c_{109}\qoo\nonumber\\
&&+c_{110}\roho+c_{111}\soho+c_{112}\ro r+c_{113}\so r+c_{114}(\po-\qqo)+c_{115}\qo+c_{116}\to+c_{117}r_{|0}+c_{118}r^2+c_{119}p\nonumber\\
&&+c_{120}q+c_{121}t\big\}\delta^i{}_k+\big\{c_{201}\roo^2+c_{202}\roo\ro+c_{203}\roo\so+c_{204}\roo r+c_{205}\rooho+c_{206}\ro^2+c_{207}\ro\so\nonumber\\
&&+c_{208}\so^2+c_{209}\qoo+c_{210}\roho+c_{211}\soho+c_{212}\ro r+c_{213}\so r+c_{214}(\po-\qqo)+c_{215}\qo+c_{216}\to\nonumber\\
&&+c_{217}r_{|0}+c_{218}r^2+c_{219}p+c_{220}q+c_{221}t\big\}\yI\yk+\big\{c_{301}\roo^2+c_{302}\roo\ro+c_{303}\roo\so+c_{304}\roo r\nonumber\\
&&+c_{305}\rooho+c_{306}\ro^2+c_{307}\ro\so+c_{308}\so^2+c_{309}\qoo+c_{310}\roho+c_{311}\soho+c_{312}\ro r+c_{313}\so +c_{314}(\po r\nonumber\\
&&-\qqo)+c_{315}\qo+c_{316}\to+c_{317}r_{|0}+c_{318}r^2+c_{319}p+c_{320}q+c_{321}t\big\}\yI\bk+c_4\yI(\roohk-\rkoho)\nonumber\\
&&+(c_{501}\roo+c_{502}\ro+c_{503}\so+c_{504}r)\yI\rko+(c_{601}\roo+c_{602}\ro+c_{603}\so+c_{604}r)\yI\sko+c_7\yI(\qok\nonumber\\
&&-2\qko)+c_{8}\yI(\rkho-2\rohk)+c_{9}\yI(\skho-2\sohk)+(c_{1001}\roo+c_{1002}\ro+c_{1003}\so+c_{1004}r)\yI\rk\nonumber\\
&&+(c_{1101}\roo+c_{1102}\ro+c_{1103}\so+c_{1104}r)\yI\sk+c_{12}\yI(\pk-\qqk)+c_{13}\yI\qk+c_{14}\yI\tk+c_{15}\yI\rhk\nonumber\\
&&+\big\{c_{1601}\roo^2+c_{1602}\roo\ro+c_{1603}\roo\so+c_{1604}\roo r+c_{1605}\rooho+c_{1606}\ro^2+c_{1607}\ro\so+c_{1608}\so^2\nonumber\\
&&+c_{1609}\qoo+c_{1610}\roho+c_{1611}\soho+c_{1612}\ro r+c_{1613}\so r+c_{1614}(\po-\qqo)+c_{1615}\qo+c_{1616}\to\nonumber\\
&&+c_{1617}r_{|0}+c_{1618}r^2+c_{1619}p+c_{1620}q+c_{1621}t\big\}\bI\yk+\big\{c_{1701}\roo^2+c_{1702}\roo\ro+c_{1703}\roo\so+c_{1704}\roo r\nonumber\\
&&+c_{1705}\rooho+c_{1706}\ro^2+c_{1707}\ro\so+c_{1708}\so^2+c_{1709}\qoo+c_{1710}\roho+c_{1711}\soho+c_{1712}\ro r+c_{1713}\so r\nonumber\\
&&+c_{1714}(\po-\qqo)+c_{1715}\qo+c_{1716}\to+c_{1717}r_{|0}+c_{1718}r^2+c_{1719}p+c_{1720}q+c_{1721}t\big\}\bI\bk+c_{18}\bI(\roohk\nonumber\\
&&-\rkoho)+(c_{1901}\roo+c_{1902}\ro+c_{1903}\so+c_{1904}r)\bI\rko+(c_{2001}\roo+c_{2002}\ro+c_{2003}\so+c_{2004}r)\bI\sko\nonumber\\
&&+c_{21}(\qok-2\qko)+c_{22}\bI(\rkho-2\rohk)+c_{23}\bI(\skho-2\sohk)+(c_{2401}\roo+c_{2402}\ro+c_{2403}\so\nonumber\\
&&+c_{2404}r)\bI\rk+(c_{2501}\roo+c_{2502}\ro+c_{2503}\so+c_{2504}r)\bI\sk+c_{26}\bI(\pk-\qqk)+c_{27}\bI\qk+c_{28}\bI\tk\nonumber\\
&&+c_{29}\bI\rhk+c_{30}\sIoho\yk+c_{31}\sIoho\bk+(c_{3201}\roo+c_{3202}\ro+c_{3203}\so+c_{3204}r)\rIo\yk+(c_{3301}\roo\nonumber\\
&&+c_{3302}\ro+c_{3303}\so+c_{3304}r)\rIo\bk+c_{34}\rIo\rko+c_{35}\rIo\rk+c_{36}\rIo\sk+(c_{3701}\roo+c_{3702}\ro+c_{3703}\so\nonumber\\
&&+c_{3704}r)\sIo\yk+(c_{3801}\roo+c_{3802}\ro+c_{3803}\so+c_{3804}r)\sIo\bk+c_{39}\sIo\sko+c_{40}\sIo\rk+c_{41}\sIo\sk\nonumber\\
&&+c_{42}\tIo\yk+c_{43}\tIo\bk+c_{44}(\sIkho-2\sIohk)+c_{45}(\rIho+\sIho)\yk+c_{46}(\rIho+\sIho)\bk+(c_{4701}\roo\nonumber\\
&&+c_{4702}\ro+c_{4703}\so+c_{4704}r)\rIk+c_{48}\tIk+(c_{4901}\roo+c_{4902}\ro+c_{4903}\so+c_{4904}r)\rI\yk+(c_{5001}\roo\nonumber\\
&&+c_{5002}\ro+c_{5003}\so+c_{5004}r)\rI\bk+c_{51}\rI\rko+c_{52}\rI\sko+c_{53}\rI\rk+c_{54}\rI\sk+(c_{5501}\roo+c_{5502}\ro\nonumber\\
&&+c_{5503}\so+c_{5504}r)\sI\yk+(c_{5601}\roo+c_{5602}\ro+c_{5603}\so+c_{5604}r)\sI\bk+c_{57}\sI\rko+c_{58}\sI\sko+c_{59}\sI\rk\nonumber\\
&&+c_{60}\sI\sk+c_{61}(\qI+\tI)\yk+c_{62}(\qI+\tI)\bk+c_{63}(\rIhk+\sIhk),
\end{eqnarray}
where $\RIk$ are the Riemann tensor of $\a$, and
\begin{eqnarray*}
&c_{101}=\f{3b^2\a^2-8b\a\b+10\b^2}{9(b\a+\b)^2(b\a-\b)^2},\quad c_{102}=\f{2\a(9b^3\a^3-38b^2\a^2\b+28b\a\b^2-9\b^3)}{9b(b\a+\b)^2(b\a-\b)^3},\quad
c_{103}=-\f{2\a(b^2\a^2+b\a\b+3\b^2)}{3b(b\a+\b)(b\a-\b)^3},\\
&c_{104}=\f{2\a^3(9b^2\a^2-16b\a\b+17\b^2)}{9b(b\a+\b)^2(b\a-\b)^3},\quad
c_{105}=-\f{b\a-2\b}{3(b\a+\b)(b\a-\b)},\quad
c_{106}=\f{\a(2b\a-\b)(21b^4\a^4-37b^3\a^3\b+54b^2\a^2\b^2-9b\a\b^3-9\b^4)}{9b^3(b\a+\b)^2(b\a-\b)^4},\\
&c_{107}=\f{2\a(2b^4\a^4+17b^3\a^3\b-6b^2\a^2\b^2-6b\a\b^3+3\b^4)}{3b^3(b\a+\b)(b\a-\b)^4},\quad
c_{108}=\f{\a(2b^3\a^3+5b^2\a^2\b-3b\a\b^2+\b^3)}{b^3(b\a-\b)^4},\quad c_{109}=\f{8\a^2(b\a-2\b)}{3(b\a+\b)(b\a-\b)^2},\\
&c_{110}=-\f{\a(4b^2\a^2-5b\a\b+3\b^2)}{3b(b\a+\b)(b\a-\b)^2},\quad c_{111}=-\f{\a\b}{b(b\a-\b)^2},\quad
c_{112}=\f{4\a^3(9b^4\a^4-41b^3\a^3\b+31b^2\a^2\b^2-6b\a\b^3-3\b^4)}{9b^3(b\a+\b)^2(b\a-\b)^4},\\
&c_{113}=-\f{4\a^3(3b^3\a^3+b\a\b^2+\b^3)}{3b^3(b\a+\b)(b\a-\b)^4},\quad c_{114}=-\f{4\a^3(b\a-2\b)}{3b(b\a+\b)(b\a-\b)^2},\quad
c_{115}=\f{4\a^3(4b^2\a^2-5b\a\b+3\b^2)}{3b(b\a+\b)(b\a-\b)^3},\quad c_{116}=\f{4\a^3\b}{b(b\a-\b)^3},\\
&c_{117}=-\f{2\a^3(b\a-2\b)}{3b(b\a+\b)(b\a-\b)^2},\quad c_{118}=\f{4\a^6(6b^2\a^2-8b\a\b+7\b^2)}{9b^2(b\a+\b)^2(b\a-\b)^4},\quad
c_{119}=-\f{2\a^4(4b^2\a^2-5b\a\b+3\b^2)}{3b^2(b\a+\b)(b\a-\b)^3},\\
&c_{120}=\f{4\a^4(2b^2\a^2-b\a\b+3\b^2)}{3b^2(b\a+\b)(b\a-\b)^3},\quad c_{121}=\f{2\a^4\b}{b^2(b\a-\b)^3},\quad
c_{201}=-\f{b(3b^3\a^3-6b^2\a^2\b+31b\a\b^2-8\b^3)}{9\a(b\a+\b)^3(b\a-\b)^3},\\
&c_{202}=-\f{2(9b^5\a^5-70b^4\a^4\b+31b^3\a^3\b^2-47b^2\a^2\b^3+26b\a\b^4-9\b^5)}{9b\a(b\a+\b)^3(b\a-\b)^4},\quad
c_{203}=\f{2(b^4\a^4+13b^3\a^3\b+14b^2\a^2\b^2-b\a\b^3+3\b^4)}{3b\a(b\a+\b)^2(b\a-\b)^4},\\
&c_{204}=-\f{2\a(6b^4\a^4-13b^3\a^3\b+57b^2\a^2\b^2-5b\a\b^3+15\b^4)}{9b(b\a+\b)^3(b\a-\b)^4},\quad
c_{205}=\f{b(b^2\a^2-4b\a\b+\b^2)}{3\a(b\a+\b)^2(b\a-\b)^2},\\
&c_{206}=-\f{3b^7\a^7+b^6\a^6\b+391b^5\a^5\b^2-305b^4\a^4\b^3
+97b^3\a^3\b^4-9b^2\a^2\b^5-27b\a\b^6+9\b^7}{9b^3\a(b\a+\b)^3(b\a-\b)^5},\\
&c_{207}=\f{2(3b^6\a^6-62b^5\a^5\b-31b^4\a^4\b^2-2b^3\a^3\b^3+3b^2\a^2\b^4+12b\a\b^5-3\b^6)}{3b^3\a(b\a+\b)^2(b\a-\b)^5},\quad
c_{208}=-\f{(b\a+\b)(3b^3\a^3+13b^2\a^2\b-7b\a\b^2+\b^3)}{b^3\a(b\a-\b)^5},\\
&c_{209}=\f{-4(3b^3\a^3-11b^2\a^2\b+2b\a\b^2+4\b^3)}{3(b\a+\b)^2(b\a-\b)^3},\quad
c_{210}=-\f{(b^3\a^3-13b^2\a^2\b+7b\a\b^2-3\b^3)\b}{3b\a(b\a+\b)^2(b\a-\b)^3},\quad
c_{211}=\f{(b\a+\b)\b}{b\a(b\a-\b)^3},\\
&c_{212}=\f{4\a(3b^6\a^6+24b^5\a^5\b-36b^4\a^4\b^2+136b^3\a^3\b^3-32b^2\a^2\b^4-24b\a\b^5+9\b^6)}{9b^3(b\a+\b)^3(b\a-\b)^5},\\
&c_{213}=\f{4\a(6b^5\a^5-9b^4\a^4\b+28b^3\a^3\b^2+23b^2\a^2\b^3-11b\a\b^4+3\b^5)}{3b^3(b\a+\b)^2(b\a-\b)^5},\quad
c_{214}=\f{2\a(b\a-3\b)(4b^2\a^2-b\a\b-2\b^2)}{3b(b\a+\b)^2(b\a-\b)^3},\\
&c_{215}=-\f{2\a^2(4b^3\a^3-11b^2\a^2\b+26b\a\b^2-7\b^3)}{3(b\a+\b)^2(b\a-\b)^4},\quad c_{216}=-\f{6\a^2\b}{(b\a-\b)^4},\quad
c_{217}=-\f{2\a(b^3\a^3-b^2\a^2\b-2b\a\b^2+6\b^3)}{3b(b\a+\b)^2(b\a-\b)^3},\\
&c_{218}=-\f{2\a^4(12b^4\a^4-17b^3\a^3\b+37b^2\a^2\b^2+9b\a\b^3+39\b^4)}{9b^2(b\a+\b)^3(b\a-\b)^5},\quad
c_{219}=\f{\a^2(8b^4\a^4-16b^3\a^3\b+25b^2\a^2\b^2-2b\a\b^3-3\b^4)}{3b^2(b\a+\b)^2(b\a-\b)^4},\\
&c_{220}=-\f{2\a^2(4b^4\a^4-2b^3\a^3\b+23b^2\a^2\b^2+2b\a\b^3-3\b^4)}{3b^2(b\a+\b)^2(b\a-\b)^4},\quad
c_{221}=-\f{\a^2(4b\a-\b)\b}{b^2(b\a-\b)^4},\quad
c_{301}=-\f{2(b^3\a^3-13b^2\a^2\b+6b\a\b^2-4\b^3)}{9(b\a+\b)^3(b\a-\b)^3},\\
&c_{302}=-\f{2\a(26b^4\a^4-27b^3\a^3\b+91b^2\a^2\b^2-39b\a\b^3+9\b^4)}{9b(b\a+\b)^3(b\a-\b)^4},\quad
c_{303}=-\f{2\a(2b^3\a^3+5b^2\a^2\b+2b\a\b^2+\b^3)}{b(b\a+\b)^2(b\a-\b)^4},\\
&c_{304}=-\f{2\a^3(b^3\a^3-45b^2\a^2\b+17b\a\b^2-33\b^3)}{9b(b\a+\b)^3(b\a-\b)^4},\quad
c_{305}=\f{2(b^2\a^2-b\a\b+\b^2)}{3(b\a+\b)^2(b\a-\b)^2},\\
&c_{306}=\f{\a(7b^6\a^6+232b^5\a^5\b-221b^4\a^4\b^2+232b^3\a^3\b^3-99b^2\a^2\b^4+9\b^6)}{9b^3(b\a+\b)^3(b\a-\b)^5},\\
&c_{307}=\f{2\a(26b^5\a^5+34b^4\a^4\b+29b^3\a^3\b^2-9b^2\a^2\b^3-3b\a\b^4+3\b^5)}{3b^3(b\a+\b)^2(b\a-\b)^5},\quad
c_{308}=\f{\a(b\a+\b)(13b^2\a^2-4b\a\b+\b^2)}{b^3(b\a-\b)^5},\\
&c_{309}=-\f{4\a^2(5b^2\a^2-5b\a\b+2\b^2)}{3(b\a+\b)^2(b\a-\b)^3},\quad
c_{310}=\f{\a(b^3\a^3-13b^2\a^2\b+7b\a\b^2-3\b^3)}{3b(b\a+\b)^2(b\a-\b)^3},\quad
c_{311}=-\f{\a(b\a+\b)}{b(b\a-\b)^3},\\
&c_{312}=-\f{\a^3(47b^5\a^5-34b^4\a^4\b+428b^3\a^3\b^2-214b^2\a^2\b^3+69b\a\b^4+24\b^5)}{9b^3(b\a+\b)^3(b\a-\b)^5},\quad
c_{313}=-\f{\a^3(9b^4\a^4+77b^3\a^3\b+51b^2\a^2\b^2+15b\a\b^3+8\b^4)}{3b^3(b\a+\b)^2(b\a-\b)^5},\\
&c_{314}=\f{2\a^3(5b^2\a^2-5b\a\b+2\b^2)}{3b(b\a+\b)^2(b\a-\b)^3},\quad
c_{315}=-\f{2\a^3(2b^3\a^3-9b^2\a^2\b+4b\a\b^2-\b^3)}{b(b\a+\b)^2(b\a-\b)^4},\quad
c_{316}=\f{2\a^3(2b\a+\b)}{b(b\a-\b)^4},\\
&c_{317}=\f{2\a^3(b^2\a^2-b\a\b+4\b^2)}{3b(b\a+\b)^2(b\a-\b)^3},\quad
c_{318}=-\f{2\a^6(b^3\a^3-35b^2\a^2\b+7b\a\b^2-53\b^3)}{9b^2(b\a+\b)^3(b\a-\b)^5},\quad
c_{319}=\f{\a^4(2b^3\a^3-9b^2\a^2\b+4b\a\b^2-\b^3)}{b^2(b\a+\b)^2(b\a-\b)^4},\\
&c_{320}=\f{2\a^4(7b^2\a^2+\b^2)\b}{b^2(b\a+\b)^2(b\a-\b)^4},\quad
c_{321}=\f{\a^4(2b\a+\b)}{b^2(b\a-\b)^4},\quad
c_{4}=\f{2(b\a-2\b)}{3(b\a+\b)(b\a-\b)},\quad
c_{501}=\f{2(2b\a-\b)\b}{9(b\a+\b)^2(b\a-\b)^2},\\
&c_{502}=-\f{\a(3b^3\a^3-b^2\a^2\b+11b\a\b^2-9\b^3)}{9b(b\a+\b)^2(b\a-\b)^3},\quad
c_{503}=-\f{\a(b^2\a^2+4b\a\b-3\b^2)}{3b(b\a+\b)(b\a-\b)^3},\quad
c_{504}=-\f{2\a^3(3b^2\a^2-4b\a\b-\b^2)}{9b(b\a+\b)^2(b\a-\b)^3},\\
&c_{601}=\f{2(2b\a-\b)\b}{(b\a+\b)(b\a-\b)^3},\quad
c_{602}=\f{\a(b^3\a^3-13b^2\a^2\b+7b\a\b^2-3\b^3)}{b(b\a+\b)(b\a-\b)^4},\quad
c_{603}=-\f{3\a(b\a+\b)^2}{b(b\a-\b)^4},\\
&c_{604}=\f{2\a^3(3b^2\a^2-7b\a\b+8\b^2)}{b(b\a+\b)(b\a-\b)^4},\quad
c_{7}=-\f{4\a^2(b\a-2\b)}{3(b\a+\b)(b\a-\b)^2},\quad
c_{8}=-\f{\a(4b^2\a^2-5b\a\b +3\b^2)}{3b(b\a+\b)(b\a-\b)^2},\quad
c_{9}=-\f{\a\b}{b(b\a-\b)^2},\\
&c_{1001}=\f{\a(3b^3\a^3-13b^2\a^2\b-19b\a\b^2+9\b^3)}{9b(b\a+\b)^2(b\a-\b)^3},\quad
c_{1002}=-\f{\a(39b^5\a^5-89b^4\a^4\b-17b^3\a^3\b^2+18b^2\a^2\b^3-36b\a\b^4+9\b^5)}{9b^3(b\a+\b)^2(b\a-\b)^4},\\
&c_{1003}=\f{\a(b^4\a^4-2b^3\a^3\b+27b^2\a^2\b^2+15b\a\b^3-3\b^4)}{3b^3(b\a+\b)(b\a-\b)^4},\quad
c_{1004}=-\f{\a^3(48b^4\a^4-115b^3\a^3\b+26b^2\a^2\b^2+141b\a\b^3-24\b^4)}{9b^3(b\a+\b)^2(b\a-\b)^4},\\
&c_{1101}=\f{\a(b^2\a^2-8b\a\b+3\b^2)}{3b(b\a+\b)(b\a-\b)^3},\quad
c_{1102}=-\f{\a(11b^4\a^4-40b^3\a^3\b+27b^2\a^2\b^2-15b\a\b^3+3\b^4)}{3b^3(b\a+\b)(b\a-\b)^4},\quad
c_{1103}=\f{\a(b^3\a^3+b^2\a^2\b+6b\a\b^2-\b^3)}{b^3(b\a-\b)^4},\\
&c_{1104}=-\f{\a^3(12b^3\a^3-45b^2\a^2\b+55b\a\b^2-8\b^3)}{3b^3(b\a+\b)(b\a-\b)^4},\quad
c_{12}=-\f{4\a^3(b\a-2\b)}{3b(b\a+\b)(b\a-\b)^2},\quad
c_{13}=-\f{2\a^3(4b^2\a^2-5b\a\b+3\b^2)}{3b(b\a+\b)(b\a-\b)^3},\\
&c_{14}=-\f{2\a^3\b}{b(b\a-\b)^3},\quad
c_{15}=\f{4\a^3(b\a-2\b)}{3b(b\a+\b)(b\a-\b)^2},\quad
c_{1601}=\f{4(b^2\a^2+4\b^2)\b}{9(b\a+\b)^3(b\a-\b)^3},\quad
c_{1602}=-\f{2\a(7b^3\a^3+43b^2\a^2\b-17b\a\b^2+27\b^3)\b}{9b(b\a+\b)^3(b\a-\b)^4},\\
&c_{1603}=-\f{2\a(7b^2\a^2+14b\a\b+9\b^2)\b}{3b(b\a+\b)^2(b\a-\b)^4},\quad
c_{1604}=\f{4\a^3(3b^2\a^2+5b\a\b+22\b^2)\b}{9b(b\a+\b)^3(b\a-\b)^4},\quad
c_{1605}=\f{2\b^2}{3(b\a+\b)^2(b\a-\b)^2},\\
&c_{1606}=\f{\a(21b^6\a^6-30b^5\a^5\b+25b^4\a^4\b^2+324b^3\a^3\b^3
-153b^2\a^2\b^4-54b\a\b^5+27\b^6)}{9b^3(b\a+\b)^3(b\a-\b)^5},\\
&c_{1607}=-\f{2\a(4b^5\a^5-24b^4\a^4\b-75b^3\a^3\b^2-3b^2\a^2\b^3+27b\a\b^4-9\b^5)}{3b^3(b\a+\b)^2(b\a-\b)^5},\quad
c_{1608}=-\f{\a(5b^3\a^3-37b^2\a^2\b+15b\a\b^2-3\b^3)}{b^3(b\a-\b)^5},\\
&c_{1609}=-\f{4\a^2(b^2\a^2-b\a\b+2\b^2)}{3(b\a+\b)^2(b\a-\b)^3},\quad
c_{1610}=-\f{\a(b^3\a^3-b^2\a^2\b-b\a\b^2+9\b^3)}{3b(b\a+\b)^2(b\a-\b)^3},\quad
c_{1611}=\f{\a(b\a-3\b)}{b(b\a-\b)^3},\\
&c_{1612}=\f{\a^3(63b^5\a^5-104b^4\a^4\b-308b^3\a^3\b^2+134b^2\a^2\b^3-75b\a\b^4-30\b^5)}{9b^3(b\a+\b)^3(b\a-\b)^5},\quad
c_{1613}=\f{\a^3(21b^4\a^4-69b^3\a^3\b-87b^2\a^2\b^2-15b\a\b^3-10\b^4)}{3b^3(b\a+\b)^2(b\a-\b)^5},\\
&c_{1614}=\f{2\a^3(2b^2\a^2-b\a\b+\b^2)}{3b(b\a+\b)^2(b\a-\b)^3},\quad
c_{1615}=\f{2\a^3(b^3\a^3+2b^2\a^2\b-3b\a\b^2+12\b^3)}{3b(b\a+\b)^2(b\a-\b)^4},\quad
c_{1616}=-\f{2\a^3(b\a-4\b)}{b(b\a-\b)^4},\\
&c_{1617}=-\f{2\a^3(2b^2\a^2-b\a\b-5\b^2)}{3b(b\a+\b)^2(b\a-\b)^3},\quad
c_{1618}=\f{8\a^6(b^2\a^2+5b\a\b+14\b^2)\b}{9b^2(b\a+\b)^3(b\a-\b)^5},\quad
c_{1619}=-\f{\a^4(5b^2\a^2-2b\a\b+9\b^2)\b}{3b^2(b\a+\b)^2(b\a-\b)^4},\\
&c_{1620}=\f{2\a^4(7b^2\a^2+8b\a\b+9\b^2)\b}{3b^2(b\a+\b)^2(b\a-\b)^4},\quad
c_{1621}=\f{3\a^4\b}{b^2(b\a-\b)^4},\quad
c_{1701}=-\f{2\a^2(b^2\a^2+9\b^2)}{9(b\a+\b)^3(b\a-\b)^3},\\
&c_{1702}=\f{4\a^3(b\a+9\b)(2b^2\a^2-b\a\b+2\b^2)}{9b(b\a+\b)^3(b\a-\b)^4},\quad
c_{1703}=\f{4\a^3(b\a+2\b)(2b\a+3\b)}{3b(b\a+\b)^2(b\a-\b)^4},\quad
c_{1704}=-\f{4\a^5(2b^2\a^2+5b\a\b+23\b^2)}{9b(b\a+\b)^3(b\a-\b)^4},\\
&c_{1705}=-\f{2\a^2\b}{3(b\a+\b)^2(b\a-\b)^2},\quad
c_{1706}=-\f{\a^3(23b^5\a^5+28b^4\a^4\b+172b^3\a^3\b^2-126b^2\a^2\b^3+45b\a\b^4+18\b^5)}{9b^3(b\a+\b)^3(b\a-\b)^5},\\
&c_{1707}=-\f{2\a^3(27b^4\a^4+35b^3\a^3\b+3b^2\a^2\b^2+9b\a\b^3+6\b^4)}{3b^3(b\a+\b)^2(b\a-\b)^5},\quad
c_{1708}=-\f{\a^3(19b^2\a^2-b\a\b+2\b^2)}{b^3(b\a-\b)^5},\quad
c_{1709}=-\f{4\a^4(b\a-3\b)}{3(b\a+\b)^2(b\a-\b)^3},\\
&c_{1710}=\f{2\a^3(b^2\a^2+3\b^2)}{3b^2(b\a+\b)^2(b\a-\b)^3},\quad
c_{1711}=\f{2\a^3}{b(b\a-\b)^3},\quad
c_{1712}=\f{2\a^5(29b^4\a^4+70b^3\a^3\b-38b^2\a^2\b^2+90b\a\b^3+9\b^4)}{9b^3(b\a+\b)^3(b\a-\b)^5},\\
&c_{1713}=\f{2\a^5(15b^3\a^3+35b^2\a^2\b+27b\a\b^2+3\b^3)}{3b^3(b\a+\b)^2(b\a-\b)^5},\quad
c_{1714}=\f{2\a^5(b\a-3\b)}{3b(b\a+\b)^2(b\a-\b)^3},\quad
c_{1715}=-\f{2\a^5(5b^2\a^2-2b\a\b+9\b^2)}{3b(b\a+\b)^2(b\a-\b)^4},\\
&c_{1716}=-\f{6\a^5}{b(b\a-\b)^4},\quad
c_{1717}=-\f{2\a^5(b\a+3\b)}{3b(b\a+\b)^2(b\a-\b)^3},\quad
c_{1718}=-\f{8\a^8(b^2\a^2+5b\a\b+14\b^2)}{9b^2(b\a+\b)^3(b\a-\b)^5},\quad
c_{1719}=\f{\a^6(5b^2\a^2-2b\a\b+9\b^2)}{3b^2(b\a+\b)^2(b\a-\b)^4},\\
&c_{1720}=-\f{2\a^6(7b^2\a^2+8b\a\b+9\b^2)}{3b^2(b\a+\b)^2(b\a-\b)^4},\quad
c_{1721}=-\f{3\a^6}{b^2(b\a-\b)^4},\quad
c_{18}=\f{2\a^2}{3(b\a+\b)(b\a-\b)},\quad
c_{1901}=-\f{2\a^2\b}{9(b\a+\b)^2(b\a-\b)^2},\\
&c_{1902}=\f{2\a^4(3b\a-\b)}{9(b\a+\b)^2(b\a-\b)^3},\quad
c_{1903}=\f{2\a^4}{3(b\a+\b)(b\a-\b)^3},\quad
c_{1904}=-\f{4\a^5\b}{9b(b\a+\b)^2(b\a-\b)^3},\quad
c_{2001}=-\f{2\a^2\b}{(b\a+\b)(b\a-\b)^3},\\
&c_{2002}=\f{4\a^3(b^2\a^2-2b\a\b+3\b^2)}{b(b\a+\b)(b\a-\b)^4},\quad
c_{2003}=\f{12\a^3\b}{b(b\a-\b)^4},\quad
c_{2004}=\f{2\a^5(3b\a-7\b)}{b(b\a+\b)(b\a-\b)^4},\quad
c_{21}=-\f{4\a^4}{3(b\a+\b)(b\a-\b)^2},\\
&c_{22}=-\f{\a^3(b\a-3\b)}{3b(b\a+\b)(b\a-\b)^2},\quad
c_{23}=\f{\a^3}{b(b\a-\b)^2},\quad
c_{2401}=-\f{2\a^3(3b^2\a^2-4b\a\b-9\b^2)}{9b(b\a+\b)^2(b\a-\b)^3},\\
&c_{2402}=-\f{\a^3(21b^4\a^4-53b^3\a^3\b+18b^2\a^2\b^2+99b\a\b^3-9\b^4)}{9b^3(b\a+\b)^2(b\a-\b)^4},\quad
c_{2403}=\f{\a^3(19b^3\a^3-24b^2\a^2\b-36b\a\b^2+3\b^3)}{3b^3(b\a+\b)(b\a-\b)^4},\\
&c_{2404}=-\f{\a^5(63b^3\a^3-46b^2\a^2\b-105b\a\b^2+12\b^3)}{9b^3(b\a+\b)^2(b\a-\b)^4},\quad
c_{2501}=-\f{2\a^3(b\a-3\b)}{3b(b\a+\b)(b\a-\b)^3},\quad
c_{2502}=-\f{\a^3(11b^3\a^3-30b^2\a^2\b+36b\a\b^2-3\b^3)}{3b^3(b\a+\b)(b\a-\b)^4},\\
&c_{2503}=\f{\a^3(5b^2\a^2-13b\a\b+\b^2)}{b^3(b\a-\b)^4},\quad
c_{2504}=-\f{\a^5(21b^2\a^2-39b\a\b+4\b^2)}{3b^3(b\a+\b)(b\a-\b)^4},\quad
c_{26}=-\f{4\a^5}{3b(b\a+\b)(b\a-\b)^2},\quad
c_{27}=-\f{2\a^5(b\a-3\b)}{3b(b\a+\b)(b\a-\b)^3},\\
&c_{28}=\f{2\a^5}{b(b\a-\b)^3},\quad
c_{29}=\f{4\a^5}{3b(b\a+\b)(b\a-\b)^2},\quad
c_{30}=-\f{2(b\a-2\b)}{(b\a-\b)^2},\quad
c_{31}=-\f{2\a^2}{(b\a-\b)^2},\quad
c_{3201}=\f{2\b^2}{3(b\a+\b)^2(b\a-\b)^2},\\
&c_{3202}=-\f{\a(b^3\a^3-b^2\a^2\b-b\a\b^2+9\b^3)}{3b(b\a+\b)^2(b\a-\b)^3},\quad
c_{3203}=\f{\a(b\a-3\b)}{b(b\a-\b)^3},\quad
c_{3204}=-\f{2\a^3(2b^2\a^2-b\a\b-5\b^5)}{3b(b\a+\b)^2(b\a-\b)^3},\quad
c_{3301}=-\f{2\a^2\b}{3(b\a+\b)^2(b\a-\b)^2},\\
&c_{3302}=\f{2\a^3(b^2\a^2+3\b^2)}{3b(b\a+\b)^2(b\a-\b)^3},\quad
c_{3303}=\f{2\a^3}{b(b\a-\b)^3},\quad
c_{3304}=-\f{2\a^5(b\a+3\b)}{3b(b\a+\b)^2(b\a-\b)^3},\quad
c_{34}=-\f{2\a^2}{3(b\a+\b)(b\a-\b)},\\
&c_{35}=-\f{\a^3(b\a-3\b)}{3b(b\a+\b)(b\a-\b)^2},\quad
c_{36}=\f{\a^3}{b(b\a-\b)^2},\quad
c_{3701}=\f{6\b^2}{(b\a+\b)(b\a-\b)^3},\quad
c_{3702}=\f{3\a(b\a-3\b)(3b^2\a^2-\b^2)}{b(b\a+\b)(b\a-\b)^4},\\
&c_{3703}=\f{9\a(b^2\a^2-4b\a\b+\b^2)}{b(b\a-\b)^4},\quad
c_{3704}=\f{12\a^3\b^2}{b(b\a+\b)(b\a-\b)^4},\quad
c_{3801}=-\f{6\a^2\b}{(b\a+\b)(b\a-\b)^3},\quad
c_{3802}=\f{6\a^3(2b^2\a^2-b\a\b+\b^2)}{b(b\a+\b)(b\a-\b)^4},\\
&c_{3803}=\f{6\a^3(2b\a+\b)}{b(b\a-\b)^4},\quad
c_{3804}=-\f{12\a^5\b}{b(b\a+\b)(b\a-\b)^4},\quad
c_{39}=\f{6\a^2(b\a-3\b)}{(b\a-\b)^3},\quad
c_{40}=-\f{3\a^3(3b\a-5\b)}{b(b\a-\b)^3},\quad
c_{41}=-\f{3\a^3(3b\a-5\b)}{b(b\a-\b)^3},\\
&c_{42}=\f{4\a^2(b\a-2\b)}{(b\a-\b)^3},\quad
c_{43}=\f{4\a^4}{(b\a-\b)^3},\quad
c_{44}=-\f{2\a^2}{b\a-\b},\quad
c_{45}=\f{\a(2b\a-3\b)}{b(b\a-\b)^2},\quad
c_{46}=\f{\a^3}{b(b\a-\b)^2},\quad
c_{4701}=\f{2\a^2}{3(b\a+\b)(b\a-\b)},\\
&c_{4702}=\f{2\a^3(b\a-3\b)}{3b(b\a+\b)(b\a-\b)^2},\quad
c_{4703}=-\f{2\a^3}{b(b\a-\b)^2},\quad
c_{4704}=\f{4\a^5}{3b(b\a+\b)(b\a-\b)^2},\quad
c_{48}=-\f{4\a^4}{(b\a-\b)^2},\\
&c_{4901}=\f{\a(b^3\a^3-b^2\a^2\b-7b\a\b^2-9\b^3)}{3b(b\a+\b)^2(b\a-\b)^3},\quad
c_{4902}=-\f{\a(25b^5\a^5-31b^4\a^4\b-69b^3\a^3\b^2+30b^2\a^2\b^3+18b\a\b^4-9\b^5)}{3b^3(b\a+\b)^2(b\a-\b)^4},\\
&c_{4903}=-\f{3\a(3b^3\a^3-9b^2\a^2\b+4b\a\b^2-\b^3)}{b^3(b\a-\b)^4},\quad
c_{4904}=\f{\a^4(5b^2\a^2-14b\a\b-27\b^2)\b}{3b^2(b\a+\b)^2(b\a-\b)^4},\quad
c_{5001}=-\f{2\a^3(b^2\a^2-3b\a\b-6\b^2)}{3b(b\a+\b)^2(b\a-\b)^3},\\
&c_{5002}=-\f{\a^3(19b^4\a^4+5b^3\a^3\b-12b^2\a^2\b^2+21b\a\b^3+3\b^4)}{3b^3(b\a+\b)^2(b\a-\b)^4},\quad
c_{5003}=-\f{\a^3(3b^2\a^2+5b\a\b+\b^2)}{b^3(b\a-\b)^4},\quad
c_{5004}=-\f{\a^6(5b^2\a^2-14b\a\b-27\b^2)}{3b^2(b\a+\b)^2(b\a-\b)^4},\\
&c_{51}=-\f{\a^3(b\a-3\b)}{3b(b\a+\b)(b\a-\b)^2},\quad
c_{52}=-\f{3\a^3(3b\a-5\b)}{b(b\a-\b)^3},\quad
c_{53}=\f{\a^3(25b^3\a^3-12b^2\a^2\b-39b\a\b^2+6\b^3)}{3b^3(b\a+\b)(b\a-\b)^3},\quad
c_{54}=\f{\a^3(9b^2\a^2-15b\a\b+2\b^2)}{b^3(b\a-\b)^3},\\
&c_{5501}=-\f{\a(b^2\a^2-2b\a\b+3\b^2)}{b(b\a+\b)(b\a-\b)^3},\quad
c_{5502}=-\f{3\a(3b^4\a^4-8b^3\a^3\b+b^2\a^2\b^2+3b\a\b^3-\b^4)}{b^3(b\a+\b)(b\a-\b)^4},\quad
c_{5503}=-\f{\a(7b^3\a^3-25b^2\a^2\b+12b\a\b^2-3\b^3)}{b^3(b\a-\b)^4},\\
&c_{5504}=\f{3\a^4(b\a-3\b)\b}{b^2(b\a+\b)(b\a-\b)^4},\quad
c_{5601}=-\f{2\a^3(b\a-2\b)}{b(b\a+\b)(b\a-\b)^3},\quad
c_{5602}=-\f{\a^3(9b^3\a^3-10b^2\a^2\b+6b\a\b^2+\b^3)}{b^3(b\a+\b)(b\a-\b)^4},\\
&c_{5603}=-\f{\a^3(3b^2\a^2+5b\a\b+\b^2)}{b^3(b\a-\b)^4},\quad
c_{5604}=-\f{3\a^6(b\a-3\b)}{b^2(b\a+\b)(b\a-\b)^4},\quad
c_{57}=\f{\a^3}{b(b\a-\b)^2},\quad
c_{58}=-\f{3\a^3(3b\a-5\b)}{b(b\a-\b)^3},\\
&c_{59}=\f{\a^3(9b^2\a^2-15b\a\b+2\b^2)}{b^3(b\a-\b)^3},\quad
c_{60}=\f{\a^3(7b\a-\b)(b\a-2\b)}{b^3(b\a-\b)^3},\quad
c_{61}=-\f{2\a^3\b}{b(b\a-\b)^3},\quad
c_{62}=\f{2\a^5}{b(b\a-\b)^3},\quad
c_{63}=-\f{2\a^3}{b(b\a-\b)}.
\end{eqnarray*}
\end{proposition}

\begin{proposition}
The Ricci curvature of singular square metric (\ref{singularBerwald}) is given by
\begin{eqnarray}\label{yanega;dngadg}
\RicooF&=&\Ricoo+C_1\roo^2+C_2\rooho+C_3\roo\ro+C_4\roo\so+C_5\roo\rIi+C_6\roo r+C_7\qoo+C_8\too+C_9\ro^2\nonumber\\
&&+C_{10}\ro\so+C_{11}\so^2+C_{12}\roohb+C_{13}\roho+C_{14}\soho+C_{15}\ro\rIi+C_{16}\so\rIi+C_{17}\ro r+C_{18}\so r\nonumber\\
&&+C_{19}\po+C_{20}\qqo+C_{21}\qo+C_{22}\to+C_{23}\sIohi+C_{24}\rohb+C_{25}\sohb+C_{26}r_{|0}+C_{27}r\rIi\nonumber\\
&&+C_{28}r^2+C_{29}\tIi+C_{30}p+C_{31}q+C_{32}t+C_{33}\rIhi+C_{34}\sIhi+C_{35}\rhb,
\end{eqnarray}
where $\Ricoo$ is the Ricci curvature of $\a$, and
\begin{eqnarray*}
&C_1=\f{(3n-5)b^2\a^2-8(n-1)b\a\b+10(n-2)\b^2}{9(b\a+\b)^2(b\a-\b)^2},\quad C_2=-\f{(n-1)b\a-2(n-2)\b}{3(b\a+\b)(b\a-\b)},\\
&C_3=\f{2\a\big\{(9n-5)b^3\a^3-2(19n-34)b^2\a^2\b+2(14n-25)b\a\b^2-9(n-1)\b^3\big\}}{9b(b\a+\b)^2(b\a-\b)^3},\\
&C_4=-\f{2\a\big\{(n-5)b^2\a^2+(n-3)b\a\b+3(n-1)\b^2\big\}}{3b(b\a+\b)(b\a-\b)^3},\quad C_5=\f{2\a^2}{3(b\a+\b)(b\a-\b)},\\
&C_6=\f{2\a^3\big\{(9n-25)b^2\a^2-2(8n-11)b\a\b+(17n-19)\b^2\big\}}{9b(b\a+\b)^2(b\a-\b)^3},\quad C_7=\f{4\a^2\big\{(2n-1)b\a-(4n-7)\b\big\}}{3(b\a+\b)(b\a-\b)^2},\quad C_8=-\f{4\a^2(b\a-3\b)}{(b\a-\b)^3},\\
&C_9=\f{\a\big\{2(21n-64)b^5\a^5-(95n-167)b^4\a^4\b
+(145n-143)b^3\a^3\b^2-36(2n-3)b^2\a^2\b^3-9(n-1)b\a\b^4+9(n-1)\b^5\big\}}{9b^3(b\a+\b)^2(b\a-\b)^4},\\
&C_{10}=\f{2\a\big\{2(n-14)b^4\a^4+(17n-15)b^3\a^3\b-6(n-7)b^2\a^2\b^2-6(n-1)b\a\b^3+3(n-1)\b^4\big\}}
{3b^3(b\a+\b)(b\a-\b)^4},\\
&C_{11}=\f{\a\big\{2(n-8)b^3\a^3+5(n+3)b^2\a^2\b-3(n-1)b\a\b^2+(n-1)\b^3\big\}}{b^3(b\a-\b)^4},\quad C_{12}=\f{2\a^2}{3(b\a+\b)(b\a-\b)},\\
&C_{13}=-\f{\a\big\{2(2n-5)b^2\a^2-(5n-7)b\a\b+3(n-1)\b^2\big\}}{3b(b\a+\b)(b\a-\b)^2},\quad C_{14}=\f{\a\big\{2b\a-(n-1)\b\big\}}{b(b\a-\b)^2},\quad
C_{15}=\f{2\a^3(b\a-3\b)}{3b(b\a+\b)(b\a-\b)^2},\\
&C_{16}=-\f{2\a^3}{b(b\a-\b)^2},\quad C_{17}=\f{2\a^3\big\{(18n-31)b^4\a^4-2(41n-101)b^3\a^3\b+62(n-2)b^2\a^2\b^2
-12(n+6)b\a\b^3-3(2n-7)\b^4\big\}}{9b^3(b\a+\b)^2(b\a-\b)^4},\\
&C_{18}=-\f{2\a^3\big\{(6n-35)b^3\a^3+13b^2\a^2\b+(2n+31)b\a\b^2+(2n-7)\b^3\big\}}{3b^3(b\a+\b)(b\a-\b)^4},\quad
C_{19}=-\f{2\a^3\big\{(2n-1)b\a-(4n-5)\b\big\}}{3b(b\a+\b)(b\a-\b)^2},\\
&C_{20}=\f{2\a^3\big\{(2n-5)b\a-(4n-5)\b\big\}}{3b(b\a+\b)(b\a-\b)^2},\quad
C_{21}=\f{2\a^3\big\{(8n-37)b^2\a^2-2(5n-11)b\a\b+3(2n+9)\b^2\big\}}{3b(b\a+\b)(b\a-\b)^3},\\
&C_{22}=-\f{2\a^3\big\{11b\a-(2n+9)\b\big\}}{b(b\a-\b)^3},\quad C_{23}=\f{4\a^2}{b\a-\b},\quad C_{24}=\f{2\a^3(b\a-3\b)}{3b(b\a+\b)(b\a-\b)^2},\\
&C_{25}=-\f{2\a^3}{b(b\a-\b)^2},\quad C_{26}=-\f{2\a^3\big\{(n-1)b\a-2(n-2)\b\big\}}{3b(b\a+\b)(b\a-\b)^2},\quad C_{27}=\f{4\a^5}{3b(b\a+\b)(b\a-\b)^2},\\
&C_{28}=\f{2\a^5\big\{(12n-55)b^3\a^3-8(2n-5)b^2\a^2\b+(14n+23)b\a\b^2-6\b^3\big\}}{9b^3(b\a+\b)^2(b\a-\b)^4},\quad C_{29}=-\f{4\a^4}{(b\a-\b)^2},\\
&C_{30}=-\f{2\a^3\big\{(4n-17)b^3\a^3-5(n-2)b^2\a^2\b+3(n+4)b\a\b^2-3\b^3\big\}}{3b^3(b\a+\b)(b\a-\b)^3},\quad
C_{31}=\f{4\a^3\big\{(2n-17)b^3\a^3-(n-8)b^2\a^2\b+3(n+4)b\a\b^2-3\b^3\big\}}{3b^3(b\a+\b)(b\a-\b)^3},\\
&C_{32}=-\f{2\a^3\big\{3b^2\a^2-(n+5)b\a\b+\b^2\big\}}{b^3(b\a-\b)^3},\quad
C_{33}=-\f{2\a^3}{b(b\a-\b)},\quad C_{34}=-\f{2\a^3}{b(b\a-\b)},\quad C_{35}=\f{4\a^5}{3b(b\a+\b)(b\a-\b)^2}.
\end{eqnarray*}
\end{proposition}
\end{document}

%% file: symbol.tex
\newcommand{\f}{\frac}

\newcommand{\pppp}[4]%
  {\frac{\partial^3{#1}}{\partial{#2}\partial{#3}\partial{#4}}}

\newcommand{\pab}{\alpha\phi\left(\frac{\beta}{\alpha}\right)}

\newcommand{\gab}{\alpha\phi\left(b^2,\frac{\beta}{\alpha}\right)}

\newcommand{\sq}{\frac{(\alpha+\beta)^2}{\alpha}}

\newcommand{\yI}{y^i}
\newcommand{\yk}{y_k}

\renewcommand{\a}{\alpha}
\renewcommand{\b}{\beta}
\newcommand{\ab}{(\alpha,\beta)}

\newcommand{\aij}{a_{ij}}

\newcommand{\bij}{b_{i|j}}

\newcommand{\bi}{b_i}
\newcommand{\bj}{b_j}
\newcommand{\bk}{b_k}
\newcommand{\bI}{b^i}

\newcommand{\rij}{r_{ij}}

\newcommand{\rIi}{r^i{}_i}

\newcommand{\rIk}{r^i{}_k}
\newcommand{\roo}{r_{00}}
\newcommand{\rko}{r_{k0}}
\newcommand{\rIo}{r^{i}{}_{0}}
\newcommand{\ri}{r_i}
\newcommand{\rj}{r_j}
\newcommand{\rI}{r^i}
\newcommand{\rk}{r_k}

\newcommand{\ro}{r_0}

\newcommand{\rIiho}{r^i{}_{i|0}}
\newcommand{\rIohi}{r^i{}_{0|i}}
\newcommand{\rIihb}{r^i{}_{i|b}}
\newcommand{\rooho}{r_{00|0}}
\newcommand{\roohk}{r_{00|k}}
\newcommand{\roohb}{r_{00|b}}
\newcommand{\rkoho}{r_{k0|0}}
\newcommand{\rkohb}{r_{k0|b}}
\newcommand{\rIhk}{r^i{}_{|k}}
\newcommand{\rIho}{r^i{}_{|0}}
\newcommand{\rIhi}{r^i{}_{|i}}

\newcommand{\rIhb}{r^i{}_{|b}}
\newcommand{\rkhb}{r_{k|b}}
\newcommand{\roho}{r_{0|0}}
\newcommand{\rohk}{r_{0|k}}
\newcommand{\rkho}{r_{k|0}}

\newcommand{\rohb}{r_{0|b}}
\newcommand{\rhI}{r_{|}{}^i}
\newcommand{\rhk}{r_{|k}}
\newcommand{\rhb}{r_{|b}}
\newcommand{\rIkho}{r^i{}_{k|0}}
\newcommand{\rIkhb}{r^i{}_{k|b}}
\newcommand{\rIohk}{r^i{}_{0|k}}
\newcommand{\rkohI}{r_{k0|}{}^{i}}
\newcommand{\rkhI}{r_{k|}{}^{i}}
\newcommand{\rIoho}{r^i{}_{0|0}}
\newcommand{\roohI}{r_{00|}{}^{i}}
\newcommand{\rIohb}{r^i{}_{0|b}}
\newcommand{\rohI}{r_{0|}{}^{i}}

\newcommand{\sij}{s_{ij}}
\newcommand{\sIk}{s^i{}_k}
\newcommand{\sIo}{s^i{}_0}
\newcommand{\sko}{s_{k0}}

\newcommand{\si}{s_i}
\newcommand{\sj}{s_j}
\newcommand{\sk}{s_k}
\newcommand{\sI}{s^i}

\newcommand{\so}{s_0}

\newcommand{\sIohi}{s^i{}_{0|i}}

\newcommand{\sIohk}{s^i{}_{0|k}}
\newcommand{\sIohb}{s^i{}_{0|b}}
\newcommand{\sIkho}{s^i{}_{k|0}}

\newcommand{\sIkhi}{s^i{}_{k|i}}

\newcommand{\sIoho}{s^i{}_{0|0}}
\newcommand{\sIhi}{s^i{}_{|i}}

\newcommand{\sIhk}{s^i{}_{|k}}
\newcommand{\skhI}{s_{k|}{}^{i}}
\newcommand{\sIho}{s^i{}_{|0}}
\newcommand{\soho}{s_{0|0}}
\newcommand{\sohb}{s_{0|b}}

\newcommand{\sohk}{s_{0|k}}
\newcommand{\sohI}{s_{0|}{}^i}
\newcommand{\skho}{s_{k|0}}

\newcommand{\co}{c_0}
\newcommand{\cb}{c_b}
\newcommand{\cI}{c^i}
\newcommand{\ck}{c_k}
\newcommand{\db}{d_b}
\newcommand{\dI}{d^i}
\newcommand{\dk}{d_k}

\newcommand{\poo}{p_{00}}
\newcommand{\pI}{p^i}
\newcommand{\pk}{p_k}
\newcommand{\po}{p_0}
\newcommand{\pIi}{p^i{}_i}
\newcommand{\pIk}{p^i{}_k}
\newcommand{\pko}{p_{k0}}
\newcommand{\pIo}{p^i{}_0}

\newcommand{\qk}{q_k}

\newcommand{\qI}{q^i}
\newcommand{\qIo}{q^i{}_0}
\newcommand{\qoo}{q_{00}}
\newcommand{\qoI}{q_{0}{}^i}
\newcommand{\qko}{q_{k0}}
\newcommand{\qok}{q_{0k}}
\newcommand{\qo}{q_0}
\newcommand{\qqk}{q^*{}_k}

\newcommand{\qqI}{q^*{}^i}
\newcommand{\qqo}{q^*{}_0}
\newcommand{\qIi}{q^i{}_{i}}

\newcommand{\too}{t_{00}}

\newcommand{\tIo}{t^i{}_{0}}
\newcommand{\tko}{t_{k0}}
\newcommand{\tk}{t_k}

\newcommand{\tI}{t^i}
\renewcommand{\to}{t_0}
\newcommand{\tIk}{t^i{}_{k}}
\newcommand{\tIi}{t^i{}_{i}}

\newcommand{\dIk}{\delta^i{}_k}

\newcommand{\RIkF}{{}^F{R}^i{}_k}
\newcommand{\RIk}{{R}^i{}_k}

\newcommand{\RIb}{{R}^i{}_b}
\newcommand{\Rbb}{R_{bb}}

\newcommand{\RbIkb}{R_{b}{}^i{}_{kb}}
\newcommand{\RoIkb}{R_{0}{}^i{}_{kb}}
\newcommand{\RbIko}{R_{b}{}^i{}_{k0}}
\newcommand{\Rbokb}{R_{b0kb}}

\newcommand{\RicooF}{{}^F{Ric_{00}}}
\newcommand{\Ricoo}{{Ric_{00}}}

\newcommand{\Rico}{{Ric_{0}}}

\newcommand{\Ric}{{Ric}}

\newcommand{\Rat}{\mathfrak{Rat}}
\newcommand{\Irrat}{\mathfrak{Irrat}}